\documentclass[11pt,reqno,tbtags,a4paper,final]{amsart}
\usepackage[a4paper,margin=1in,marginparwidth=2cm]{geometry}

\usepackage[backend=bibtex,maxbibnames=99,style=alphabetic,isbn=false,url=false,doi=false]{biblatex}
\renewbibmacro*{volume+number+eid}{%
	\printfield{volume}%
	\setunit*{\addnbspace}
	\printfield{number}%
	\setunit{\addcomma\space}%
	\printfield{eid}}
\DeclareFieldFormat[article]{number}{\mkbibparens{#1}}
\usepackage{amsmath,amssymb,amsfonts,amsthm,bbm}
\usepackage{xspace,xpunctuate}
\usepackage[notcite,notref]{showkeys}

\usepackage{url}
\usepackage{tikz-cd}
\usepackage{xcolor}
\usepackage{enumerate}
\usepackage{tikz,float}
\usepackage{subcaption,graphicx}
\usepackage{mathtools}
\usepackage[pdfencoding=auto,psdextra,breaklinks]{hyperref}
\usepackage[capitalise]{cleveref}
\usepackage{bookmark}
\usepackage{crossreftools}
\crefname{equation}{}{}
\Crefname{claim}{Claim}{Claims}

\let\tp\texorpdfstring%

\newcommand{\ignore}[1]{\relax}

\definecolor{brown}{cmyk}{0, 0.72, 1, 0.45}
\definecolor{grey}{gray}{0.5}

\definecolor{lightRed}{cmyk}{0, 0.3, 0.3, 0.0}


\numberwithin{equation}{section}
\allowdisplaybreaks%

\DeclarePairedDelimiter{\set} {\{} {\}}
\DeclarePairedDelimiter{\card}{|}{|}
\DeclarePairedDelimiter{\abs}{|}{|}
\newcommand{\ceil}[1]{\left\lceil{#1}\right\rceil}
\newcommand{\floor}[1]{\left\lfloor{#1}\right\rfloor}
\newcommand{\parens}[1]{\left({#1}\right)}

\theoremstyle{plain}
\newtheorem{theorem}{Theorem}[section]
\newtheorem{lemma}[theorem]{Lemma}

\newtheorem{conjecture}[theorem]{Conjecture}
\newtheorem{claim}[theorem]{Claim}
\newtheorem{remark}[theorem]{Remark}

\theoremstyle{definition}

\theoremstyle{remark}

\DeclareMathOperator{\E}{\mathbb{E}} 
\DeclareMathOperator{\Exp}{Exp} 
\DeclareMathOperator{\Bi}{Bi} 
\DeclareMathOperator{\Bern}{Bernoulli} 
\DeclareMathOperator{\Prob}{\mathbb{P}} 
\renewcommand{\Pr}{\Prob}
\newcommand{\Prr}[1]{\Pr\left({#1}\right)}
\newcommand{\Med}{\operatorname{Med}}

\newcommand{\whp}{w.h.p\xperiod} 
\newcommand{\Whp}{W.h.p\xperiod} 

\newcommand{\tend}{\longrightarrow}

\newcommand\pto{\overset{\mathrm{p}}{\tend}}

\newcommand\downto{\searrow}

\newcommand\eqdef{\stackrel{\operatorname{def}}{=}}

\newcommand{\eps}{\varepsilon}

\newcommand{\la}{\lambda}

\newcommand{\OO}[1]{O\left({#1}\right)}

\renewcommand{\k}{k}

\newcommand{\rad}{\operatorname{rad}}

\newcommand{\iid}{i.i.d\xperiod}
\newcommand{\st}{$s$--$t$\xspace}
\newcommand{\ste}{\ensuremath{\set{s,t}}\xspace}

\renewcommand{\a}{\alpha}
\newcommand{\tmbf}[1]{\textbf{\boldmath{#1}}}

\newcommand{\kcost}{\frac{2\k + \ln n}{n}}

\newcommand{\RR}{\ensuremath{R}\xspace}
\newcommand{\dd}{d} 

\newcommand{\astar}{a^\star}
\newcommand{\eab}{\set{a,b}}
\newcommand{\Xab}{X_{a,b}}
\newcommand{\lfrac}[2]{#1/#2}

\newcommand{\rr}{r_}
\newcommand{\Rk}{R^{(k)}}
\newcommand{\Rsub}{\rho_}
\newcommand{\kk}{k+\rr0}
\newcommand{\kp}{{k+1}}
\newcommand{\km}{{k-1}}
\newcommand{\nm}{{n-1}}
\newcommand{\os}[1]{_{(#1)}}
\newcommand{\Wo}[1]{W\os{#1}}
\newcommand{\Wok}{W\os k}

\newcommand{\Wokp}{W\os{\kp}}
\newcommand{\Wokk}{W\os{\kk}}
\newcommand{\Woi}{W\os i}
\newcommand{\Uok}{U\os k}
\newcommand{\kbar}{\bar{k}}
\renewcommand{\log}{\ln}
\newcommand{\wo}{w_0}
\newcommand{\ope}{(1+\eps)}
\newcommand{\dist}{\operatorname{dist}}

\newcommand{\Bk}{{B_k}}
\newcommand{\epsk}{\eps_k}

\newcommand{\SkLower}{(1-\eps)\sum_{i=1}^{k}\left(\frac{2i+\ln n}{n}\right)}

\newcommand{\Klower}{\sqrt{\ln n}}
\newcommand{\Khigher}{n-1}

\newcommand{\KrangeHigh}{\ln^{11/10} n}
\newcommand{\nto}{\ln^3 n}
\newcommand{\sumkk}{\sum_{\kbar=1}^{n/2}}
\newcommand{\kstar}{{k^\star}}
\newcommand{\SE}{^{(E)}}
\newcommand{\SU}{^{(U)}}
\renewcommand{\asymp}{\sim}
\newcommand{\eqnote}[1]{\quad\text{(#1)}}
\newcommand{\Ce}{C_\eps}
\newcommand{\CB}{C_B}
\newcommand{\stebound}{n^{0.01}}
\newcommand{\evp}[1]{\mathcal{P}_{#1}}
\newcommand{\evl}[1]{\mathcal{L}_{#1}}
\newcommand{\eva}{\mathcal{A}}

\makeatletter
\newcommand\llabel[1]{&\refstepcounter{equation}(\theequation)\ltx@label{#1}&}
\makeatother

\newcommand{\Bp}{B_k}
\newcommand{\epsp}{\eps_k}
\newcommand{\expkstar}{n-\sqrt{n}}
\newcommand{\RD}{L}
\newcommand{\expcutoff}{n/2}
\newcommand{\Ukbaseval}{3n^{-2/10}}
\renewcommand{\Klower}{\sqrt[3]{\ln n}}
\newcommand{\EWW}[1]{\E \Wo{#1}}

\newcommand\ER{Erd\H os--R\'enyi\xspace}

\title[Successive shortest paths]{Successive shortest paths in \\ complete graphs with random edge weights}
\date{10 October 2020}

\author[Stefanie Gerke]{Stefanie Gerke}
\address[Stefanie Gerke]{Department of Mathematics,
Royal Holloway University of London,
Egham Hill, Egham TW20 0EX, England}
\email{stefanie.gerke@rhul.ac.uk}

\author{Bal\'azs F. Mezei}
\address[Bal\'azs F. Mezei]{
Department of Computer Science, University of Oxford,
Wolfson Building,
Parks Road,
Oxford
OX1 3QD,
England
\textnormal{
(Previously, for most of the project, 
Department of Mathematics,
Royal Holloway University of London,
Egham Hill, Egham TW20 0EX, England.)}
}
\email{balazs.mezei@cs.ox.ac.uk}

\author[Gregory B. Sorkin]{Gregory B. Sorkin}
\address[Gregory B. Sorkin]{Department of Mathematics,
The London School of Economics and Political Science,
Houghton Street, London WC2A 2AE, England}
\email{g.b.sorkin@lse.ac.uk}

\begin{document}

\begin{abstract}
Consider a complete graph $K_n$ with edge weights
drawn independently from a uniform distribution $U(0,1)$.
The weight of the shortest (minimum-weight) path $P_1$ between two given vertices
is known to be $\ln n / n$, asymptotically.
Define a second-shortest path $P_2$ to be the shortest path edge-disjoint from $P_1$, and consider more generally the shortest path $P_k$ edge-disjoint from all earlier paths.
We show that the cost $X_k$ of $P_k$
converges in probability to $2k/n+\ln n/n$
uniformly for all $k \leq n-1$.
We show analogous results when the edge weights are drawn from an
exponential distribution.
The same results characterise the collectively cheapest $k$
edge-disjoint paths, i.e., a minimum-cost $k$-flow.
We also obtain the expectation of $X_k$ conditioned on the existence of $P_k$.
\vspace*{-2em}
\end{abstract}

\maketitle

\section{Introduction}

It is a standard problem to find the shortest \st path in a graph,
i.e., the cheapest path $P_1$ between specified vertices $s$ and $t$,
and its cost $X_1$,
where the cost of a path is the sum of the costs of its edges.
We will use the terms ``cost'' and ``weight'' interchangeably,
and reserve ``length'' for the number of edges in a path.

Consider the complete graph $G=K_n$ with each edge $\set{u,v}$ having weight $w(u,v)$,
where the $w(u,v)$ are i.i.d.\ random variables with exponential distribution $\Exp(1)$
or uniform distribution $U(0,1)$ (we consider both versions).
In this random setting, a well-known result of Janson~\cite{Janson123} is that as $n \to \infty$,
\begin{equation}\label{eq:svante-X1}
	\frac{X_1}{\ln n / n } \pto 1 .
\end{equation}

We define the second cheapest path $P_2$, with cost $X_2$, to be the cheapest \st path edge-disjoint from $P_1$,
and in general define $P_k$, with cost $X_k$, to be the cheapest \st path edge-disjoint from
$P_1 \cup \cdots \cup P_{k-1}$,
provided such a path exists.
We also think of this as finding path $P_k$ after the
preceding paths' edges have been removed.
Our question is how the costs $X_k$ behave in the limit as $n \to \infty$
(this limit is implicit throughout).
Our main result is the following.

\begin{theorem}\label{Tmain}
In the complete graph $K_n$ with \iid uniform $U(0,1)$ edge weights,
with $X_k$ the cost of the $k$th cheapest path,
\begin{align}
 \frac{X_k}{2k/n + \ln n / n} & \pto 1 \label{eq:UnifMain}
\end{align}
uniformly for all $k \leq n-1$.
That is, for any $\eps>0$, asymptotically almost surely, for every $k =1,\ldots,n-1$,
\begin{align}\label{Xkbounds}
   1-\eps &\leq
     \frac{X_k}{2k/n + \ln n / n}
     \leq 1+\eps .
\end{align}
\end{theorem}
Naturally, with $k=1$, \cref{eq:UnifMain} recovers
Janson's result \cref{eq:svante-X1},
since $2/n=o(\ln n/n)$.

As discussed shortly, in contrast to many cases,
the result for the uniform distribution does not extend immediately to
all distributions with positive density at 0.
However, we have a corresponding result for exponentially distributed edge weights.
Given an edge-weight distribution, let $\Wok$ be the (random)
weight of the $k$th cheapest edge out of a vertex
(the $k$th order statistic of $n-1$ edge weights).

\begin{theorem}\label{Texp}
In the complete graph $K_n$ with \iid exponential edge weights with mean 1,
\begin{align}
  \frac{X_k}{2 \E \Wok + \ln n / n} & \pto 1 \label{eq:ExpMain}
\end{align}
uniformly for all $k \leq n-1$.
\end{theorem}

\noindent We give the guiding intuition behind the formula \cref{eq:ExpMain} in \cref{sec:paths-intution}.
Note that $\E \Wok =\sum_{i=1}^k \tfrac{1}{n-i}$ in the exponential case (see e.g. \cref{lemma:edge-orderstat}).
In the uniform case, $\E \Wok = k/n$,
so \cref{eq:UnifMain} in \cref{Tmain} can also be written as \cref{eq:ExpMain}.

Rather than finding the $k$ successive cheapest paths, we may alternatively
wish to find the $k$ edge-disjoint paths of \emph{collective} minimum cost.
Equivalently, where every edge of $G$ has capacity~$1$,
we may be interested in the minimum-cost $k$-flow from $s$ to $t$ in $G$.
The following remark shows that this problem leads to essentially the same costs.
(The analogous ``collective'' problem for minimum spanning trees is solved in \cite{FrJo},
and \cite{JaSoMST} shows that for MSTs, the ``successive'' version leads to strictly larger costs.)

\begin{remark}\label{rmk:pathsFk}
	In the complete graph $K_n$ with i.i.d.\ edge weights with distribution $U(0,1)$ or
	exponential with mean 1, the minimum-cost $k$-flow has cost $F_k$ satisfying
	\begin{align} \label{kflow}
	\frac{F_k}{
		\sum_{i=1}^k
        (2 \E W_{(i)}+ \ln n/n)
	} \pto 1
	\end{align}
	uniformly for all $k \leq n-1$.
\end{remark}
As in \cref{Xkbounds}, the statement consists of high-probability upper and lower bounds.
The upper bounds here, for the two models,
follow immediately from the upper bounds of \cref{eq:UnifMain} and \cref{eq:ExpMain}.
The lower bounds follow from the lower bound on $S_k \coloneqq \sum_{i=1}^k X_k$ (see \cref{Skdef})
in \cref{Sklower} and its analogue for the exponential case, as
those bounds hold for any set of $k$ edge-disjoint paths.
(The main work in \cref{lowerbound}, not needed here, is to extract lower bounds on $X_k$
from the lower bounds on $S_k$.)

\begin{remark}\label{rmk:existence}
$P_k$ is always defined for all $k \leq n/2$,
but, at least for $n$ even,
may be undefined for all $k>n/2$.
\end{remark}
\begin{proof}
There are $n-2$ length-2 \st paths.
Any path $P_k$ can destroy (share an edge with) at most two such paths
(since $P_k$ uses just one edge incident to each of $s$ and $t$).
Also, the single-edge path \ste is destroyed only by
the path $P_k$ consisting of just this edge.
So, for $P_1,\ldots,P_{k}$ to destroy all length-1 and length-2 paths requires
$k \geq (n-2)/2+1=n/2$,
so for $k \leq n/2$, certainly path $P_k$ exists.
	
Conversely, a construction described in 1892 by Lucas \cite[pp.~162--164]{Lucas},
which he attributes to Walecki,
shows that a complete graph $K_{2r}$ can be decomposed into
$r$ edge-disjoint Hamilton paths whose $2r$ terminals are all distinct.
For $n$ even, decompose $G=K_n \setminus \set{s,t}$ in this way,
then link $s$ to one ``start'' terminal of each such path and $t$ to the other ``end'' terminal,
giving $(n-2)/2$ edge-disjoint \st paths.
The edge \ste gives another path, for $n/2$ paths in all.
The only edges not used by these paths are a star from $s$ to the Hamilton paths' end terminals,
and another star from $t$ to their start terminals,
and as there are no other unused edges to connect these two stars,
there is no further \st path.
With nonzero probability, the edge weights are such that
$P_1,\ldots,P_{n/2}$ are these $n/2$ paths,
so that $P_{n/2+1}$ does not exist.
\end{proof}

\cref{rmk:existence} implies that, at least for $n$ even,
$\E[X_k]$ is undefined for $k > \expcutoff$.
The following theorem establishes $\E X_k$ for $k \leq \expcutoff$,
and for all $k \leq n-1$,
gives the expectation
conditioned on the (high-probability) event that $P_k$ exists.

\begin{theorem}\label{thm:expectation}
In both the uniform and exponential models,
for $k \leq n-1$, a.a.s.\ $P_k$ exists, and
\begin{align}
\E[X_k \mid P_k \textnormal{ exists}] &=(1+o(1))(2\E \Wok + \ln n / n), \label{EX}
\end{align}
uniformly in $k$.
\end{theorem}
For $k \leq \expcutoff$, by \cref{rmk:existence} the conditioning is null,
so it is immediate from \cref{thm:expectation}
that $E[X_k]=(1+o(1))(2\E \Wok + \ln n / n)$.

\subsection{Intuition}\label{sec:paths-intution}
The intuitive picture is that
path $P_k$ should use the $k$th cheapest edges out of $s$ and $t$, whose costs are denoted $\Wok^s$ and $\Wok^t$ respectively.
Then, if we ignore previous paths' use of other edges in $G\setminus \set{s,t}$,
by \cref{eq:svante-X1} the opposite endpoints of these two edges should be connected by a path of cost about $\ln n/n$.
This suggests that $X_k \leq \Wok^s+\Wok^t + \ln n/n$,
and this is our guiding intuition.
Obviously, the path $P_k$ does not have to use the $k$th cheapest edge, its middle section may cost more or less than $\ln n / n$, and as earlier paths use up edges, the costs of these middle sections may rise.
It is true, though, that
$\sum_{i=1}^k X_i \geq \sum_{i=1}^{k-1} \parens{\Wok^s + \Wok^t}$
(summing only to $\km$ on the right-hand side to avoid doubly counting edge \ste),
and we use this in proving the lower bounds on $X_k$
(in \cref{lowerbound} for uniform and \cref{ExpLB} for exponential)
and, more surprisingly, in proving the upper bounds on $X_k$ for large $k$
(in \cref{largekUB} generically, the details treated in
\cref{unifUB,ExpBounds}).

Our upper bounds are obtained by reasoning as follows.
Janson~\cite{Janson123} analyses the shortest \st path,
and shortest-path tree (SP tree or SPT) on $s$, in the randomly edge-weighted graph $G=K_n$,
showing that the cost of $P_1$ is
asymptotically almost surely, almost exactly
$\ln n/n$.
When the path $P_1$ is deleted, this prunes away a root-level branch of the SP tree.
The SP tree is a uniform random tree,
and using known properties of such trees (see for example~\cite{Su})
it is not hard to show that what remains of the SP tree is likely to be large;
capitalising on this we can find an almost equally cheap path $P'_2$.
This line of argument also shows that there remains a cheap path after deleting $P'_2$,
but we need to know what happens when we delete the true second-shortest path $P_2$,
and at this point the argument fails because it
gives no characterisation of $P_2$, only of $P_2'$.
We do know, however, that $P_2$ is cheap (no more expensive than $P'_2$),
and of course uses just one edge incident to each of $s$ and $t$,
and we will show that deleting \emph{any} {edge set} with these properties
(including $P_2$ as a possibility) must still leave a cheap path $P'_3$, and so forth.
This ``adversarial'' deletion argument is developed in \cref{adversary} to prove \cref{Tmain}.

\subsection{Context}

The question fits with a broad research theme on optimisation (and satisfiability) problems
on random structures.
The novel element here is the ``robustness'' aspect of finding cheap structures
even after the cheapest has been removed, and in this
we were  motivated by a recent study by Janson and Sorkin~\cite{JaSoMST} of the same question for successive minimum spanning trees (MSTs), again for $K_n$ with
uniform or exponential random edge weights.
The results for shortest paths and MSTs are dramatically different.
For MSTs, it is a celebrated result of Frieze~\cite{FriezeMST} that as $n \to \infty$ the cost of the MST $T_1$ satisfies $w(T_1) \pto \zeta(3) \eqdef \sum_{k=1}^{\infty} 1/k^3$,
and~\cite{JaSoMST} shows that each subsequent tree's cost has $w(T_k) \pto \gamma_k$
with the $\gamma_k$ strictly increasing
(and $2k-2\sqrt k <\gamma_k<2k+2\sqrt k$).
That is very different from the case here, for paths,
where for $k=o(\ln n)$ we have $X_k$ asymptotically equal to $X_1$.

Further context is given
in the discussion of open problems
in \cref{otherModels}.

\subsection{Edge weight distributions}
As remarked earlier, in many contexts
(including for the length $X_1$ of a shortest path)
the result for any distribution with
positive density at 0 follows immediately from that for the uniform distribution $U(0,1)$, but that is not the case for the successive paths considered here.

\begin{remark}\label{remark:blackbox}
Janson proves the $X_1$ case in the exponential model but provides standard
``black-box'' reasoning that it
holds also for the uniform distribution, for any distribution with density 1 at 0
(i.e., with cumulative distribution function (CDF) $\Pr(X \leq x) = x+o(x)$ for $x \downto 0$),
and, after simple rescaling, for any distribution with positive density at 0.
Simply, if there is a path of cost $o(1)$ in some such model,
each edge $w$ must also cost $o(1)$, and,
coupling with the uniform distribution by replacing $w$ with $w'=F(w)$,
with $F$ the CDF,
$w' \leq (1+o(1)) w$,
and thus the same path is similarly cheap in the uniform model.
By the same token, if a path is cheap in any model, the same path has
asymptotically the same cost in any other model,
and thus the cheapest paths have asymptotically the same cost.
\end{remark}

\begin{remark}\label{remark:openbox}
In our setting this argument does not apply:
to find path $P_k$ we must know the nature of the $k-1$ previous paths;
their costs are not enough.
For $k=o(n)$, however, the standard argument applies within our proofs,
since the proofs rely only on edges of cost $o(1)$.
However, for larger $k$ there are genuine difficulties.
Our argument for the exponential case, in \cref{ExpBounds},
largely parallels that for uniform but requires new calculations
for the upper bound,
and one new idea for the lower bound (in \cref{ExpLB}).
It is not clear for what other edge-weight distributions
(even those with density 1 at 0)
\cref{eq:ExpMain} will hold.
\end{remark}

\section{Open problems}
\subsection{Poisson multigraph model}

The issue of possible non-existence of paths $P_k$ for $k>n/2$
(see \cref{rmk:existence})
is obviated if, as in \cite{JaSoMST}, we
work in a Poisson multigraph model.
Here, each pair of vertices $\set{u,v}$ of $K_n$ is joined by
infinitely many edges, whose weights are drawn from a Poisson process of rate 1
(so that the cheapest $\set{u,v}$ edge has exponentially distributed cost of mean 1).
By construction, in this model every \st path is always available (possibly at a higher cost).

\begin{conjecture}
In the Poisson multigraph model,
$ \frac{X_k}{2k/n + \ln n / n} \pto 1 $ uniformly for all $k \leq n-1$,
and $ \frac{\E X_k}{2k/n + \ln n / n} \to 1 $ for all $k \leq n-1$.
\end{conjecture}
\noindent Actually, in this model there is no need to stop at $k=n-1$,
but it is not clear how far out we can go (especially preserving uniform convergence).

\subsection{Other models} \label{otherModels}
Most narrowly,
it would be interesting to characterise successive shortest paths that are vertex-disjoint rather than edge-disjoint,
and (in the style of \Cref{rmk:pathsFk} for edge-disjoint paths)
the $k$ vertex-disjoint paths of collective minimum cost.
In this model, guessing that path lengths stay around $\log n$,
we would expect $P_k$ to be defined up to $k$ about $n / \log n$.

More broadly, it would be interesting to explore
different edge-weight distributions, different structures,
and different graphs.

As noted earlier, we have results for
uniformly and exponentially distributed edge weights,
but not for arbitrary distributions.
As mentioned, results
for the single shortest path
follow by standard arguments for any distribution
with positive density near~0.
For a distribution with density tending to 0 or $\infty$ at 0,
shortest paths were studied in \cite{BH2012}.
In particular, they consider the case when edge weights are \iid and have the same distribution as $Z^{p}$, where $Z \sim \Exp(1)$ and $p>0$ is a fixed parameter;
in this setting, the shortest path has length $p \ln n$
and its cost is $\ln n/n^p$ times a $p$-dependent constant.
A variant where the edge-weight distribution may depend on $n$ is studied in \cite{Eckhoff}.

To what distributions does \cref{Texp} extend?
Restricting to distributions with positive density near~0,
the arguments in \cref{ExpBounds} should immediately extend for all $k=o(n)$.
For larger $k$, the ``middle'' of each path should remain short,
so the issue is the edges incident on $s$ and $t$ in $P_k$.
Certainly \cref{eq:ExpMain} will fail if the order statistics
of edges incident to $s$ are not concentrated,
for example if the edge distribution is
a mixture of $U(0,1)$ and an atom at 2
or (for a continuous example)
a mixture of $U(0,1)$ and the Pareto distribution with CDF $1-1/x$ for $x\geq 1$.
It might be true that \cref{eq:ExpMain} holds more generally if
the expectation $2\E \Wok$ is replaced by $\Wok^s+\Wok^t$.
However, to obtain the needed lower bound for
the exponential model (see \cref{ExpBounds}),
we had to address the fact that the $k$th path does not
necessarily use the edges of cost $\Wok^s$ and $\Wok^t$;
we also needed exponential-specific calculations for the upper bound.

One could explore other structural models.
Minimum spanning trees (MSTs) have already been explored in
\cite{JaSoMST} for the successive version
and in \cite{FrJo} for the collective version.
But for many other models
the single cheapest structure is well studied
but the successive and collective extensions have not been explored:
this includes
perfect matchings in complete bipartite graphs $K_{n,n}$
\cite{AldousAP,WastlundAP},
perfect matchings in complete graphs $K_n$
\cite{WastlundPM},
and Hamilton cycles (i.e., the Travelling Salesman Problem) in $K_n$
\cite{Was2010}.

One could also consider graphs other than complete graphs,
in the style of studies of the MST in a random regular graph
\cite{BFM1998},
and of first-passage percolation
in \ER random graphs \cite{Bhamidi}
and hypercubes \cite{Martinsson}.

\section{Upper bound for small \tp{$k$}{k}}\label{sec:k-small}

In this section we prove the upper bound of \cref{Tmain} for all $k=o(\sqrt n)$;
larger values are treated in the next section.

As discussed in the introduction, we can characterise the cheapest path $P_1$ and subsequent
paths that are \emph{cheap} but not necessarily \emph{cheapest},
putting us at a loss to characterise what remains on deletion of a subsequent \emph{cheapest} path.
We address this in this section.
Given $k$, we show a construction of a subgraph $\RR=\Rk$
of $G$
designed so that, as we will show in turn,
its \st paths are all cheap,
and no deletion of edges from $\RR$ subject to certain constraints can destroy all these paths.
We show that the union of the $k$ shortest paths satisfies these constraints,
so that there remains a cheap \st path in $\RR$ and thus in $G$,
and use this to prove \cref{Tmain}.

\begin{figure}[h]
\centering
\includegraphics[width=0.7\linewidth]{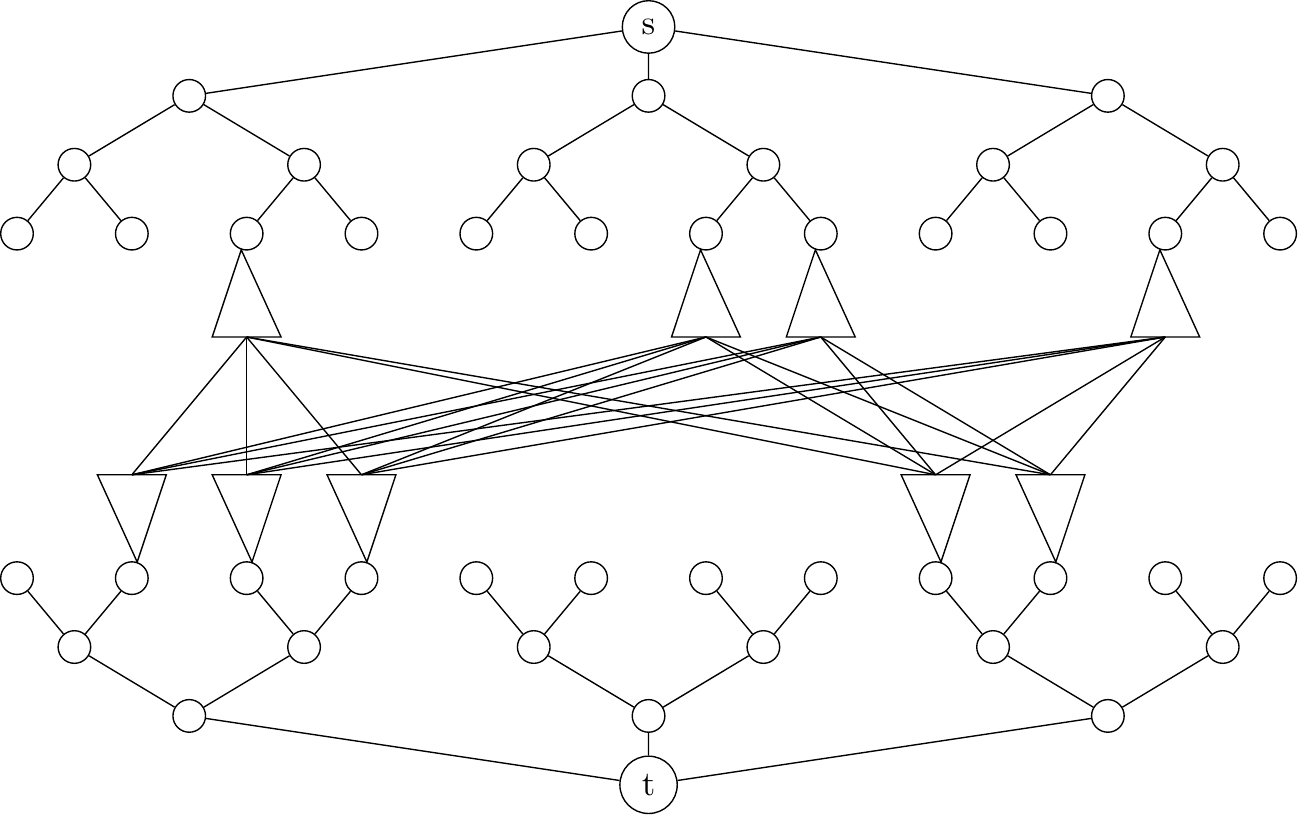}
\caption{Cartoon of a robust subgraph $\RR$ of $G$,
showing the vertices $s$ and $t$, their respective
structures $R_s$ and $R_t$
including shortest-path trees represented by triangles
(some ``failed'' and thus not shown),
and the cheap edges connecting triangles in $R_s$ and $R_t$.
Vertices $s$ and $t$ have down-degree (number of children) $r_0$,
and vertices at levels 1 and 2 (in $R_s$ and $R_t$) have down-degrees $r_1$ and $r_2$ respectively.
}\label{Frobust}
\end{figure}

Specifically, we will define a structure $\RR$,
sketched in \cref{Frobust},
that has many cheap and spread-out paths between $s$ and $t$,
within which we will always find a cheap path.
A crucial point is that
each step of the construction occurs in a complete induced subgraph of $G$
of size $n-o(n)$ with all edges unconditioned.

We will show,
assuming that
\begin{align}\label{indHyp}
X_i \leq \ope \parens { \frac{2i}n+\frac{\ln n}n }
\end{align}
for all $i\leq k$, that the same holds for $i=\kp$.
We will do so by showing that after deleting $k$ paths, each of cost $\leq \ope (2k/n+\ln n/n)$ from $G$,
some or all of whose edges may lie in $R$,
there remains a path in $R$ satisfying the same cost bound, and so this must also be true of $P_\kp$.

\medskip

Consistent with this approach,
and because to prove convergence in probability it suffices to
consider an arbitrarily small, fixed $\eps$
(see around \cref{Xkbounds}),
throughout this section we assume that $\eps>0$ is fixed.
Thus, in the $n \to \infty$ limit implicit throughout,
\begin{align}
 \eps=\Theta(1) ,  \label{epsconst}
\end{align}
and $\eps$ (and functions of $\eps$)
may be absorbed into the constants implicit
in any Landau-notation expression.

\begin{remark} \label{warning}
Most of the calculations below hold for any $\eps>0$,
but a few (\cref{Bheavy} and \cref{pathlength} for example)
hold only for $\eps$ sufficiently small.
This is not restrictive here, in proving convergence in probability,
but to characterise expectation, \cref{expSmallk}
requires $\eps$ to be a large constant
(to assure sufficiently small failure probabilities).
The proof of \cref{lem:large-eps} addresses the changes needed.
\end{remark}

Before going into detail let us sketch the construction of $R$.
We first build up a tree $R_s$ on $s$, starting from $s$ at level 0,
the opposite endpoints of edges out of level $i$ forming level $i+1$.
We will always choose ``cheap'' edges, but not always the cheapest ones, as explained later.
From $s$ we will choose $\kk$ cheap edges;
from each of these $\kk$ level-1 vertices we choose $\rr1$ cheap edges;
from each of the $(\kk)\rr1$ level-2 vertices we choose $\rr2$ cheap edges;
and on each of the $(\kk)\rr1 \rr2$ level-3 vertices we construct a shortest-path tree comprising $\dd$ vertices.
We do a similar construction on $t$ to form $R_t$.
Finally, we link $R_s$ and $R_t$ using cheap edges between their shortest-path trees.
The values of the parameters $\rr0$, $\rr1$, $\rr2$ and $d$ are given in
~\cref{k0},~\cref{k1},~\cref{k2} and~\cref{ddef},
and it is confirmed in \cref{sec:Rsize} that
the construction uses only a small fraction of $G$'s vertices,
\begin{align}\label{n'}
\card{V(R)} &= O((\kk) \rr1 \rr2 d) = o(n) ,
\end{align}
a fact we rely on in the construction.

We will repeatedly use the following Chernoff bound, which in fact holds under more general conditions;
see for example~\cite[Theorem 1, eq.~(4)]{Janson2001}.
\begin{lemma}\label{lemma:BinDev}
Let $X \sim \Bi(n,p)$  be a binomial random variable with mean $\la = np$.
Then for any $\eps>0$, $\Pr(X< (1-\eps)\la) \leq \exp(-\eps^2 \la/2)$.
\end{lemma}

\subsection{Cheap paths are short}
We show that, \whp, every cheap path in $G$ is also short.
The following lemma asserts the contrapositive.
The result is used in~\cref{Bany} to restrict the number of edges the adversary can delete.

\begin{lemma}\label{LLenBd}
In both the uniform and exponential models,
with probability $1-\OO{n^{-1.9}}$,
simultaneously for all $l$ with $\ln n \leq l < n$,
every \st path of length $l$ has cost
$\geq l/(19 n)$.
\end{lemma}

\begin{proof}
We start with the uniform distribution.
Here, with $X=\sum_{i=1}^{l} X_i$, $X_i \sim U(0,1)$ \iid,
$X$ has the Irwin-Hall distribution and it is a standard result that
$\Pr(X \leq a) \leq a^l/l!$
(see for example \cite[eq.\ 8]{SorkinClique}).
Recall that Stirling's approximation is also a lower bound.
Thus,
\begin{align*}
  \Prr{X \leq \frac{l}{19n}}
   &\leq \frac{(l/19n)^l}{l!}
   \leq \frac{(l/19n)^l}{\sqrt{2\pi l} \, (l/e)^l}
    < \parens{\frac{e}{19n}}^l.
\end{align*}
The cost of a fixed path of length $l$ has the same law as $X$.
Over the $\leq n^l$ choices for such a path, the number $M_l$ of ``cheap paths'' (of cost $<l/(19n)$) satisfies (by Markov's inequality)
\[ \Pr(M_l>0)
 \leq \E M_l
 \leq n^l \Prr{X \leq \frac{l}{19n}}
 \leq n^l \parens{\frac{e}{19n}}^l
 = \parens{\frac e{19}}^l . \]
Summing over $l \geq \ln n$, the probability that there is a cheap
path of any such length
is $\OO{{(e/19)}^{\ln n}} = \OO{n^{-1.9}}$.

Since an $\Exp(1)$ random weight $X'$ can be obtained from a $U(0,1)$ r.v.\ $X$
by setting $X' = -\ln(1-X)>X$,
the exponential weight stochastically dominates the uniform,
so the result for uniform immediately implies that for exponential.
\end{proof}

\subsection{Adversarial edge deletions}\label{adversary}
As noted in the introduction, we introduce an edge-deleting adversary
whose powers allow it to delete the paths $P_1,\ldots,P_k$,
but which is more easily characterised than those paths are.
We now specify what the adversary is permitted to do.

Let
\begin{align}\label{sdef}
 s = s(k) & \coloneqq 2k+\ln n .
\end{align}
(From context it should be easy to distinguish this use of $s$
from that as the source of an \st path.)
Let $\wo$ be the ``target cost'' of a path, namely
\begin{align}\label{w0def}
  \wo = \wo(k) & \coloneqq \frac s n = \frac{2k}n+\frac{\ln n}n .
\end{align}
Define a ``heavy'' edge to be one of cost
\begin{align}\label{heavydef}
  & \geq \tfrac1{11} \eps \wo.
\end{align}

Assuming that
each of $P_1,\ldots,P_k$
has weight $\leq \ope w_0$,
the \emph{number of heavy edges} in $P_1 \cup \cdots \cup P_k$ is at most
\begin{align}
  \frac{k \ope \wo}{\tfrac1{11} \eps \wo}
   < \frac{12k}{\eps} 
   < \frac{12 s}{\eps} . \label{Bheavy}
\end{align}
Also,
modulo the one-time failure probability $\OO{n^{-1.9}}$ from \cref{LLenBd},
by that lemma each path has length at most
\begin{align}
 \ope \wo \cdot 19n < 20 s . \label{pathlength}  
\end{align}
Thus, the length of all $k$ paths taken together
(i.e., the number of edges in $P_1 \cup \dots \cup P_k$)
is at most
\begin{align}
  20 k s < 10 s^2 . \label{Bany}
\end{align}
And of course the $k$ paths include
\begin{align} \label{Bincident}
 &\text{exactly $k$ edges incident on each of $s$ and $t$.}
\end{align}

Subject to these assumptions
--- that each of $P_1,\ldots,P_k$ has weight $\leq \ope w_0$
and that the high-probability conclusion of \cref{LLenBd} holds ---
$P_1 \cup \cdots \cup P_k$ satisfies all three of the constraints
\cref{Bheavy}, \cref{Bany}, and \cref{Bincident}
on heavy edges, all edges, and ``incident'' edges.
An adversary who can delete any edge set subject to these constraints
is able to delete $P_1 \cup \cdots \cup P_k$, which is all we require.
However, to simplify analysis we will give the adversary even more power.

At the root of $R$ we will allow the adversary to delete edges subject only to \cref{Bincident};
at level 1, additional edges subject only to
the ``heavy-edge budget'' \cref{Bheavy};
and at levels 2 and 3 and for ``middle'' edges,
additional edges subject only to
the ``edge-count budget'' \cref{Bany}.

\medskip

We will show how to choose the parameters of $R$
so that every \st path in $R$ has cost $\leq \ope \wo$, and
so that $R$ is ``robust'':
after the adversarial deletions,
at least one path remains.
Specifically, we will arrange that there remains a path in which the
``root'' edge incident to $s$ costs $\leq \tfrac k n + \frac19 \eps w_0$,
the edge out of level 1 is heavy but has cost $\leq \frac19 \eps \wo$,
the edge out of level 2 may be light or heavy and also has cost $\leq \frac19 \eps \wo$,
the path through the SP tree has total cost $\leq \frac12 \tfrac {\ln n} n + \frac19 \eps \wo$,
the central edge joining this to the opposite SP tree adds cost $\leq \frac19 \eps \wo$,
and the continuation of this path to $t$ has the symmetrical properties.
It is immediate that such a path has total cost $\leq (2k+\ln n)/n + 9 \cdot \tfrac19 \eps \wo = \ope \wo$.
(But see \cref{Rpathcost} for confirmation, after the construction is detailed.)

\medskip

\subsection{Level 0, cheapest edges}\label{level0}
On $s$, add to $R$ the $\kk$ edges of lowest cost, excluding $\set{s,t}$ from consideration,
with
\begin{align}\label{k0}
\rr0 = \ceil{\tfrac1{10} \eps s} = \Theta(s) .
\end{align}
Consider this step a \emph{failure} if any selected edge has cost greater than $\tfrac k n+\frac19 \eps \wo$.
There are $n'=n-2=(1-o(1))n$ edges under consideration, with weights \iid $U(0,1)$,
and failure occurs iff the number $X$ of edges with weights in the interval $[0, \tfrac k n+\frac19 \eps \wo]$
is smaller than $\kk$.
Note that $X \sim \Bi(n', \tfrac k n+\frac19 \eps \wo)$, thus $\E X = (1-o(1)) \, (k+\frac19 \eps s)$,
and failure means that $X <\kk$, i.e., that
\begin{align*}
\frac{X}{\E X}
 &< (1+o(1)) \, \frac{\kk}{k+\frac19 \eps s}
 = (1+o(1)) \, \frac{k+\frac1{10} \eps s}{k+\frac19 \eps s} ,
\intertext{which by $s>2k$ is}
 &< (1+o(1)) \, \frac{k+\frac1{10} \eps \cdot 2k}{k+\frac19 \eps \cdot 2k}
 = (1+o(1)) \, \frac{(1+\frac2{10} \eps)k}{(1+\frac29 \eps)k}
 < 1-\tfrac1{50}\eps 
 = 1-\Omega(\eps) .
\end{align*}
By \cref{lemma:BinDev}, then, the probability of failure is
\begin{align}
\Pr(X < (1-\Omega(\eps)) \E X)
 &\leq \exp(-\Omega(\eps^2) \E X/2)
 \leq \exp(-\Omega(\eps^2 \cdot \eps s))
 \leq \exp(-\Theta(s)) , \label{level0fail}
\end{align}
the final expression using that $\eps$ is constant (see \cref{epsconst}).

So, modulo the given failure probability, every selected edge incident on $s$ has cost $\leq \tfrac k n+\frac19 \eps \wo$,
and after the adversarial deletion of $k$ of these edges, $\rr0$ remain.
The selection of edges conditions the costs of the other edges incident on $s$,
but none will play any role in the analysis.

The purpose of the next two levels is to expand the number of edges
to the point where the adversary cannot delete all of them,
because of the heavy-edge budget \cref{Bheavy} for edges out of level 1, and the edge-count budget \cref{Bany} for edges out of level 2 and beyond.
At the same time, we try to minimise the number of vertices introduced into the construction
so that it will remain $o(n)$ for as large a $k$ as possible.

\subsection{Level 1, cheapest heavy edges}\label{level1}
From each neighbour $v$ of $s$ along the edges just added,
add to $R$ the
\begin{align}\label{k1}
\rr1 \coloneqq \ceil{10,000/\eps^2}= \Theta(1)
\end{align}
cheapest \emph{heavy} edges from $v$ to any of the $n'=n(1-o(1))$ vertices not yet added (see \cref{n'}),
as before also excluding vertex $t$.
Consider this step a \emph{failure} if any added edge has cost greater than $\frac19 \eps \wo$.
For each neighbour $v$ there are $n'$ edges under consideration, with weights \iid $U(0,1)$,
and failure occurs iff the number $X$ of edges with weights in the interval
$[\frac1{11} \eps \wo, \frac19 \eps \wo]$
is smaller than $\rr1$.
Note that $X \sim \Bi(n', (\frac19-\frac1{11}) \eps \wo)$,
thus
$\E X
 = (1-o(1)) \, (\frac19-\frac1{11}) \eps s
 = \Theta(\eps s)$.
Failure means that $X < \rr1 < \E X/2$,
so by \cref{lemma:BinDev} the probability of failure for a given $v$ is $\leq \exp(-\Theta(\eps s))$.
The number of level-1 vertices $v$ is $\kk = O(s)$,
so by the union bound the probability of any failure is
\begin{align}\label{level1fail}
  \leq O(s) \exp(-\Theta(\eps s))
   &
   \leq \exp(-\Theta(s)),
\end{align}
by suitable adjustment of the constants implicit in $\Theta$.

This edge selection conditions the costs of the other edges incident on each $v$,
but none will play any role in the analysis.
The adversary must leave $\rr0$ edges out of the root,
expanding to \[ \rr0 \rr1 \geq \tfrac1{10} \eps s \cdot 10,000/\eps^2 = 1,000 s/\eps \] (heavy) edges out of level 1,
of which (by~\cref{Bheavy}) he can delete at most $12 s/\eps$,
leaving (very generously calculated) at least

\begin{align}\label{K1}
\Rsub1 & \coloneqq 120 s/\eps = \Theta(\rr0 \rr1)
\end{align}
edges out of level 1.
The vertices at the opposite endpoints of these edges constitute level 2.

\subsection{Level 2, cheapest edges}\label{level2}
From each level 2 vertex $v$ in turn,
add to $R$ the cheapest $\rr2$ edges to any of the $n'=n(1-o(1))$ vertices not yet added,
again also excluding vertex $t$ from consideration.
Here choose $\rr2$ so as to make
\begin{align}\label{K2}
  \Rsub2 \coloneqq \Rsub1 \rr2 = 12 s^2 ,
\end{align}
namely taking
\begin{align}
  \rr2 = \frac{12 s^2}{\Rsub1} = \frac{12 s^2}{120 s/\eps} = \tfrac1{10} \eps s = \Theta(\eps s) .  \label{k2}
\end{align}
Consider this step a \emph{failure} if any added edge has cost greater than $\frac19 \eps \wo$.
For each neighbour $v$ there are $n'=(1-o(1))n$ edges under consideration, with weights \iid $U(0,1)$,
and failure occurs iff the number $X$ of edges with weights in the interval
$[0, \frac19 \eps \wo]$
is smaller than $\rr2$.
Note that $X \sim \Bi(n', \frac19 \eps \wo)$,
thus $\E X = (1-o(1)) \, \frac19 \eps s = \Theta(\eps s)$.
Failure means that $X < \rr2 < 0.99 \E X$,
so by \cref{lemma:BinDev} the probability of failure for a given $v$ is $\leq \exp(-\Theta(\eps s))$.
The number of level-2 vertices $v$ is $(\kk) \rr1 = \OO{s}$
so by the union bound the probability of any failure is
\begin{align}\label{level2fail}
  & \leq \exp(-\Theta(\eps s))
\end{align}

This edge selection conditions the costs of the other edges incident on each $v$,
but none will play any role in the analysis.
The adversary had to leave at least $\Rsub1$ edges out of level 1,
expanding to $\Rsub1 \rr2 = \Rsub2 = 12 s^2$ edges out of level 2,
of which by~\cref{Bany} he can delete at most $10 s^2$,
leaving at least $2 s^2$ edges out of level 2.
The vertices at the opposite endpoints of these edges constitute level 3.

\subsection{Level 3, shortest-path trees}\label{level3}
We now grow each level-3 vertex $v$ to a tree $T_v$ with $d$ vertices,
including $v$, choosing
\begin{align}
  d \coloneqq \ceil{ \sqrt{\frac{n \ln n}{2 s^3}} \; } < \sqrt n. \label{ddef}
\end{align}
We grow these trees one after another,
always working within the $n'=n(1-o(1))$ vertices not yet added,
and again always excluding vertex t from consideration.

Controlling the lengths of the paths in $T_v$ would allow a choice of $d$ as large as $\sqrt n$,
but we make it smaller to keep the number of vertices in $R$ as small as possible
(and thus keep it to $o(n)$ for $s$ as large as possible).

Here it will be convenient to work with exponentially rather than uniformly distributed edge weights.
There are various easy ways to arrange this.
We do so by temporarily replacing each uniform weight $w$ with a weight $w' = -\ln (1-w)$;
it is standard that these transformed weights are exponentially distributed,
and that $w' \geq w$.
We construct a shortest-path tree (SPT) of order $d$ using the transformed weights;
it will not be an SPT
for the original weights,
but its paths will be short under the original weights, which is all that we care about.

Define the distance $\dist(u,v)$ between two vertices
to be the cost of a minimum-weight path between them,
and define the radius $\rad(T_v)$ of an SPT $T_v$
to be the maximum distance
from $v$ to any vertex in $T_v$.
The radius is described by the following claim,
which we phrase in a generic setting with $n$ vertices and a root vertex $s$.

\begin{claim} \label{treeradius}
In a complete graph $K_n$ with \iid exponential edge weights with mean 1,
the radius $X=\rad(T_s)$ of a shortest-path tree $T_s$ of order $d$ is
\begin{align}\label{treeX}
X = \sum_{i=1}^{d-1} X_i ,
\end{align}
where the $X_i$ are independent random variables with $X_i \sim \Exp(i(n-i))$.
\end{claim}

\begin{proof}
Following~\cite{Janson123}, think of the process of finding shortest paths from $s$
to other vertices as first-passage percolation or ``infection spreading'' starting from $s$.
Let $\RD\coloneqq \RD(r)$
be the set of vertices within radius (distance) $r$ of $s$;
we think of gradually increasing $r$, starting with $r=0$ where $\RD=\set s$.
It is well known that each edge $(v,u) \in \RD(r) \times (V\setminus \RD(r))$
has exponentially distributed weight $W$ conditioned by $W+\dist(s,v) \geq r$,
and that these random weights are independent.
This can be seen by imagining that infection has spread to radius $r$ from $s$,
including to the vertex $v$ and additionally a length $r-\dist(s,v)$ further along the edge $(v,u)$,
and appealing to the memoryless property of the exponential distribution;
it can also be verified by analysing Dijkstra's algorithm in this randomised setting.

It follows that the distance $X_1$ to the vertex nearest $s$
is distributed as $X_1 \sim \Exp(n-1)$;
the additional distance to the next vertex is $X_2$ with $X_2 \sim \Exp(2(n-2))$ and independent of $X_1$
(for total distance $X_1+X_2$);
and when there are $i$ vertices in the tree, the additional distance to the next is $X_i \sim \Exp(i(n-i))$,
with all the $X_i$ independent, for total distance as claimed.
\end{proof}

\emph{We will only use trees} $T_v$ whose radius
is  $X \leq (1+\tfrac29 \eps) \, \tfrac12 \ln n/n$
$< \tfrac12 \ln n/n + \tfrac19 \eps \wo$.
Call a tree a failure (and do not include it in the structure $R$)
if $X > (1+\tfrac29 \eps) \tfrac12 \ln n/n$.
Declare the construction of level 3
a \emph{failure} if more than $0.01 s^2$ trees fail.

Since~\cref{treeX} is monotone increasing in $d$,
the larger the $d$, the greater the probability of failure,
so in the next paragraphs we will pessimistically take $d$ to be $\sqrt n$
(ignoring integrality since $\sqrt n$ is large).
In this case,
applying \cref{treeradius} to $T_v$, constructed in a complete graph of order
$n'=n(1-o(1))$, the expectation of $X$ is
\begin{align}
\mu \coloneqq
 \E X
 &= \sum_{i=1}^{d-1} \frac1{i(n'-i)}
 = \frac{1+o(1)}{n} \sum_{i=1}^{d-1} \frac1{i}
 = (1+o(1)) \frac{\ln d}n
 = (1+o(1)) \frac12 \frac{\ln n}n    \label{treeDia}
 .
\end{align}
Thus, failure of $T_v$ implies that
\begin{align}\label{treefail1}
  \frac X \mu &> 1+\frac15 \eps .
\end{align}
To bound the probability of this event we require one more lemma
(also used later in proving \cref{lemma:edge-orderstat}).

\begin{lemma}[{\cite[Theorem 5.1]{JansonExpTail}}]\label{exptail}
	Let $X=\sum_{i=1}^{n} X_i$ with $X_i \sim \Exp(a_i)$ independent rate-$a_i$ random variables, where $a_i \geq 0$. Write $a_{*} \coloneqq \min_{i} a_i$ and $\mu \coloneqq \E X = \sum_{i=1}^n \frac 1{a_i}$.
Then:
\newline
for any $\lambda = 1+\eps > 1$,
\begin{align}\label{exp.1}
  \Prob(X \geq \lambda \mu)
  & \leq \lambda^{-1} e^{-a_{*} \mu (\lambda - 1 - \ln \lambda)}
  \leq \exp(-\Omega(\a_* \mu))
\end{align}
for any $\lambda = 1- \eps < 1$,
\begin{align}\label{exp.2}
  \Prob(X \leq \lambda \mu)
   & \leq e^{-a_{*} \mu (\lambda - 1 - \ln \lambda)}
   \leq \exp(-\Omega(\a_* \mu)),
\end{align}
and for any $\eps>0$,
\begin{align}\label{exp.3}
  \Prob( \abs{X-\mu} \geq \eps \mu)
   & \leq 2 \exp(-\Omega(\a_* \mu )).
\end{align}
The constants implicit in the $\Omega(\cdot)$ expressions are positive and only depend on $\eps$.

\end{lemma}
\begin{proof}
The inequalities in \cref{exp.1} and \cref{exp.2} in terms of $\lambda$
are directly from~\cite[Theorem 5.1]{JansonExpTail}.
The remaining expressions, including \cref{exp.3},
follow immediately.
\end{proof}

From \cref{exp.1} of \cref{exptail},
the probability
of the event in \cref{treefail1} (and thus that of $T_v$ failing) is at most
\begin{align}
  \Pr\big( {X-\mu} > \frac15 \eps \mu \big)
    & \leq \exp(-\Omega(\a_* \mu) )
    = \exp( -\Omega(\ln n)), \label{treeFail}
\end{align}
using that $a_* = n' = (1-o(1))n$, $\mu$ is given by \cref{treeDia},
and $\eps=\Theta(1)$.

The total number of trees built is $N=(\kk) \rr1 \rr2$,
which, with reference to \cref{sdef}, \cref{k0}, \cref{k1}, and \cref{k2},
is $\Theta(s^2)$.
By \cref{treeFail}, each tree independently fails with at most some probability $p=o(1)$.
Thus, the number of trees surviving dominates $\Bi(N,1-p)$,
with expectation $\lambda=N(1-p)=N(1-o(1))$.
Failure at level 3 means that at least $0.01 s^2=\Theta(N)$ trees fail,
equivalently the number surviving is at most some $\lambda(1-\Theta(1))$,
which by \cref{lemma:BinDev} has probability
\begin{align}\label{level3fail}
  & \exp(-\Omega(s^2)) .
\end{align}

\begin{remark}\label{rem1}
When construction of a tree $T_v$ rooted at a level-3 vertex $v$ is finished,
the edge between any vertex $a$ of $T_v$ and any vertex $b$ in $V' \setminus V(T_v)$ has weight
$w(a,b)$ that
--- still in the uniform model with edge weights temporarily transformed to be exponentially distributed ---
is exponentially distributed conditional upon being $\geq \rad(T_v)-d(v,a)$.
Equivalently, the edge $\eab$ gives a $v$-to-$b$ path (through $a$)
with cost
$\rad(T_v)+\Xab$,
where the ``excess'' $\Xab$ has simple exponential distribution $\Xab \sim \Exp(1)$ (with no conditioning).
Furthermore, the $\Xab$ are independent, over all choices of $a$ and $b$.
\end{remark}

Call $R_s$ the now-complete construction on $s$.
Note that there is no conditioning on edges between the remaining vertices;
in particular, the SPT infection process (or equivalently Dijkstra's algorithm) as described in \cref{treeradius} never looked at edges between uninfected vertices.

\subsection{Symmetric construction on vertex \tp{$t$}{t}}
Just as we have constructed $R_s$, we now make a similar construction $R_t$ for vertex $t$,
with the same branching factors out of levels 0, 1, and 2 and similar SPTs on level-3 vertices.
Since the number $n'$ of vertices available after constructing $R_s$ still satisfies $n' = (1-o(1))n$,
and because the construction on $s$ did not look at nor condition any edge between these vertices,
the construction on $t$ enjoys the same properties as that on $s$.

\subsection{Edges between the trees on \tp{$s$}{s} and \tp{$t$}{t}}
It remains only to complete paths between $s$ and $t$,
which we do by adding cheap edges (where present)
between the SPTs in $R_s$ and those in $R_t$.

Let $T_u$ be an SPT rooted at a level-3 vertex $u$ of $R_s$,
and $T_v$ one rooted at a level-3 vertex $v$ of $R_t$.
Let $a$ and $b$ be any vertices in $T_u$ and $T_v$ respectively.
By \cref{rem1},
edge $\eab$ gives a $u$-to-$b$ path with cost $\rad(T_u)+\Xab$,
the collection of all the excesses $\Xab$ being \iid each with distribution $\Xab \sim \Exp(1)$.
Thus, $\eab$ gives a $u$-to-$v$ path with cost $\leq \rad(T_u) + \Xab + \rad(T_v)$.

Select, and add to the full construction $R$, any such ``middle edge'' $\eab$
having $\Xab \leq \frac19 \eps \wo$.
This completes the construction of $R$.

\subsection{Order of \tp{$R$}{R}, failure probability, and path costs}\label{sec:Rsize}
It is worth first confirming that the construction uses, as claimed, $o(n)$ vertices.
The number of vertices used is of order $(k+\rr0)\rr1\rr2 d$,
which by \cref{k0}, \cref{k1}, \cref{k2}, and \cref{ddef} is $O(s^2 d)$.
Recalling from~\cref{ddef} that $d = \ceil{ \sqrt{\lfrac{n \ln n}{2 s^3}} \;}$,
as long as the ceiling function does not affect the order of $d$,
the total number of vertices is $O(s^2 d) = O(\sqrt{n s \ln n})$,
which is $o(n)$ for $s=o(n/\ln n)$.
However, the ceiling function does affect the order of $d$ when $n \ln n/2s^3 < 1$,
i.e., when $s> {(\frac12 n \ln n)}^{1/3}$;
in this case, $d=1$, the total number of vertices used is $O(s^2)$,
and this is still $o(n)$ if $s=o(\sqrt n)$.
Taking the two cases together, the construction is valid up to any $s=o(\sqrt n)$,
or equivalently for any $k=o(\sqrt n)$.

Failures at levels 0, 1, and 2 each occur w.p.\ $\leq \exp(-\Theta(s))$
(by \cref{level0fail}, \cref{level1fail}, and \cref{level2fail}),
and at level 3 w.p.\ $\leq \exp(-\Omega(s^2))$ (by \cref{level3fail}),
so by the union bound the probability of any failure is $\leq \exp(-\Theta(s))$.

We now confirm that,
assuming that the construction was successful,
any \st path in $R$ \emph{through successful SPTs}
has cost $\leq (1+\eps) \wo$.
(Remember that there may be some unsuccessful SPTs.)
By assumption of success, any level-0 edge on $s$ or $t$ has cost
$\leq \tfrac k n+\frac19 \eps \wo$,
any level-1 edge has cost $\leq \frac19 \eps \wo$,
and any level-2 edge also has cost $\leq \frac19 \eps \wo$.
Each successful level-3 tree $T$ in $R_s$ or $R_t$ has radius
$\rad(T) \leq (1+\frac29 \eps) \frac12 \ln n/n \leq \frac12 \ln n/n + \frac19 \eps \wo$,
and each selected middle edge $\eab$ connects the roots of two trees
at an excess cost (above the sum of the two radii) of $\Xab \leq \frac19 \eps \wo$.
The total of the 9 upper bounds in question is
\begin{align}\label{Rpathcost}
 2 \cdot \frac k n + 2 \cdot \frac12 \ln n/n + 9 \cdot \frac19 \eps \wo
  &= (1+\eps) \wo .
\end{align}

\subsection{Robustness of \tp{$R$}{R}}
We now show that, after the deletion of the $k$ cheapest paths in $G$,
there remains at least one path in $R$ (that uses successful SPTs).
Recall from \cref{adversary}
that deletion of the $k$ cheapest paths in $G$ is conservatively modeled as
an adversarial deletion subject to:
\cref{Bincident}, the deletion of exactly $k$ edges incident on each of $s$ and $t$;
\cref{Bheavy}, the number of heavy edges deleted at level 1;
and \cref{Bany}, the total number of edges deleted elsewhere in $R$
(at levels 2 and 3, and joining $R_s$ and $R_t$).

Without loss of generality we may assume that the adversary does not
delete an edge within an SPT $T$,
nor a middle edge from such a tree to a facing one,
since deleting the level-2 edge into the level-3 root of $T$
destroys more paths in $R$ at the same budgetary cost.

By the assumption of success,
there are at most $0.01 s^2$ failed SPTs on each of $s$ and $t$,
and for simplicity we will deal with them by imagining all trees to be successful
but allowing the adversary his choice of this many SPTs to delete;
by the argument above we can model this as deletion of edges into the roots of these trees,
and simply add $0.02 s^2$
to this budget.

Let us now allow the adversary to
delete $k$ edges from each of $s$ and $t$,
$12s/\eps$ edges out of level 1 for each (double-counting the heavy-edge budget),
and $10.02 s^2$ edges out of level 2 for each (again double-counting).
Can he destroy all \st paths?
We have not yet made any high-probability structural assertion about the middle edges,
so this is a probabilistic question:
what is the probability, over the randomness still present in the middle edges,
that there is an adversarial deletion destroying all paths?

Of the $\kk=O(s)$ edges on $s$, the adversary chooses $k$ to delete;
there are at most $2^{O(s)}$ ways to do so.
Any choice leaves $\Theta(\rr0 \rr1) = \Theta(s)$ edges out of level 1,
of which the adversary is able to delete a positive fraction, again in at most $2^{O(s)}$ ways.
Any choice leaves $\Theta(s^2)$ edges out of level 2,
of which the adversary is able to delete a positive fraction, in at most $2^{O(s^2)}$ ways.
The adversary makes a similar set of choices on $t$,
but still this comes to just $2^{O(s^2)}$ possible outcomes in all.

A given deletion choice destroys all paths precisely if it leaves no middle edge
of excess $\leq \frac19 \eps \wo$.
(Remember that, w.l.o.g., we have excluded deletions in and between the
SPTs at level~3.)
By construction, any deletion choice leaves $\Theta(s^2)$ edges out of level 2
and thus, by \cref{ddef}, $\Theta(s^2 d) = \Omega(\sqrt{n s \ln n})$ vertices in SPTs in each of $R_s$ and $R_t$,
for $\Omega(n s \ln n)$ potential middle edges.
A middle edge is \emph{selected} if its excess cost (in the exponential model) is
$w'=-\ln(1-w) \leq \frac19 \eps \wo$,
i.e., if $1-w \geq \exp(-\frac19 \eps \wo)$,
thus is \emph{rejected} with probability $\exp(-\frac19 \eps \wo)$.
There is no path only if every potential edge is rejected,
which happens w.p.\ $\leq \exp(-\frac19 \eps \wo \cdot n s \ln n) = \exp(-\Omega(s^2 \ln n))$.
Taking the union bound over all adversarial choices,
the probability than any choice leaves no paths is
\begin{align} \label{smallkAdversaryFailure}
  2^{O(s^2)} \exp(-\Omega(s^2 \ln n)) &= \exp(-\Omega(s^2 \ln n)) .
\end{align}
This is dominated by the failure probabilities $\exp(-\Theta(s))$ for other steps.

\subsection{Success for each \tp{$k$}{k}, and for all \tp{$k$}{k}}\label{sec:small-success}
We have shown that, for any $k=o(\sqrt n)$,
subject to an absence of failures,
we can generate a robust structure $\Rk$
in which, after adversarial deletions,
there remains an \st path of cost $\leq (1+\eps) \wo(k)$.
(Remember that $\wo$ and $s$ are simple functions of $k$, per~\cref{sdef} and~\cref{w0def}. Here we retain the argument $k$ we usually suppress.)
There are two types of failures possible.
The first is
that the graph fails \cref{LLenBd}'s conclusion that ``cheap paths are short'';
this occurs w.p.\ $\OO{n^{-1.9}}$.
The second is that $\Rk$ is not robust;
this occurs w.p.\ $\OO{\exp(-\Omega(s(k)))}$.

Assume success in generating $\Rk$.
We claim that
$P_1,\ldots,P_\kp$ all have cost $\leq (1+\eps) \wo(k)$
(call this ``cheap'').
Suppose not. Then there is some $i\leq k$ for which
$P_1,\ldots,P_i$ are cheap but $P_{i+1}$ is not.
Our adversary's budget allows it to delete $P_1,\ldots,P_i$,
and by assumption of success this leaves a cheap path $P$ in $\Rk$.
Thus there is a cheap $i+1$st path in $G$, a contradiction.

It follows that \emph{for each} $k$, $X_\kp \leq (1+\eps) \wo(k)$
with probability
\begin{align}
1-O(n^{-1.9}) -O(\exp(-\Omega(s(k)))) . \label{ksucceeds}
\end{align}

\medskip

A simple calculation shows that
\whp $X_\kp \leq (1+\eps) \wo(k)$
\emph{simultaneously for all $k$} in this range,
proving the upper bound in \cref{Xkbounds}.
By the union bound, the probability of
failure to build a robust structure for \emph{any} $k$ is at most
\begin{align}
 \sum_{k=0}^{\infty} \exp(-\Omega(s(k)))
  &\leq \ln n \exp(-\Omega(\ln n)) + \sum_{k=\ln n}^{\infty} \exp(-\Omega(k)) \notag \\
  &= \exp(-\Omega(\ln n))
  = n^{-\Omega(1)}. \label{eq:all-k-small}
\end{align}
Including the probability of failure in applying \cref{LLenBd},
the total failure probability is $O(n^{-1.9} + n^{-\Omega(1)}) = o(1)$.

\subsection{Limitation to small \tp{$k$}{k}} \label{small-k-limitation}
We have established \cref{Tmain} up to any $k=o(\sqrt n)$,
and the construction of $\Rk$ was tailored to such values.
For levels using heavy edges, fanout is limited
to $O(s)$.
On the other hand, the meet-in-the-middle argument requires that each side
grow large, to $\Omega(\sqrt{n/s})$.
Thus, for small $k$,
a more-than-constant number of levels is needed.
Summing heavy edges over this many levels would exceed the target weight $(1+\eps) \wo$,
so light edges are needed.
The adversary may delete $\Theta(s^2)$ light edges,
so the construction must contain at least this many.
The construction explicitly required each light edge to lead to a new vertex,
and we do not readily see how to do otherwise as long as we are using shortest-path trees,
thus intrinsically limiting $s$ (thus $k$) to $O(\sqrt n)$.
For larger $k$, however, we can obtain sufficient heavy-edge fanout in constant depth,
permitting a simpler construction described in \cref{largekUB}.

\section{Edge order statistics}\label{sec:order-stat}

In this section we establish results on order statistics
needed in later sections.
Let $\set{\Wok}_{k=1}^{n-1}$ be the order statistics of $n-1$ \iid random variables,
variously uniform $U(0,1)$ or exponential $\Exp(1)$.
We choose $n-1$ rather than $n$ as the parameter both
because many expressions are more natural in this parametrisation,
and because this way
$\Wok$ is
the cost of the $k$th cheapest edge incident to a fixed vertex $v \in K_n$.

The following lemma is used in \cref{pathweightsdetail}.
\begin{lemma}\label{intervals}
	Let $l=n^{-0.99}$. Consider the unit interval $[0,1]$ with $n$ points placed uniformly and independently at random. Then \whp every interval of length at least $l'\geq l$ contains at least $0.99 l' n$ points.
\end{lemma}

\begin{proof}
Partition the unit interval into contiguous intervals $I_i$
each of length $L\coloneqq l/1000$,
using $\floor{1/L}$ such intervals
(possibly leaving a small interval near 1 not covered).
Any interval $I$ of length $l' \geq l$ has at least a $998/1000$ fraction of
its length covered by intervals $I_i \subset I$,
and we will show that \whp every interval $I_i$ contains at least $0.999 L n$ points
(that is, at least a $0.999$ fraction of the expectation).
If so, it follows that $I$
has at least $0.999 \cdot 0.998 l' n \geq 0.99 l' n$ points.
	
The distribution of the number of points in each interval $I_i$ of length $L$ follows the binomial distribution $\Bi(n, L)$.
By \cref{lemma:BinDev},
\begin{align*}
\Prob \parens{\Bi(n, L) \leq 0.999 L n} \leq \exp(-\Omega(L n)) ,
\end{align*}
where the sign in the $\Omega$ is taken as positive.
The probability that any interval $I_i$ contains less than $0.999$ points is, by the union bound, at most,
\begin{align}\label{eq:intervals}
\floor{1/L} \cdot \exp(-{\Omega(L n)})
= \exp(-\Omega(n^{0.01}))
= o(1)
\end{align}
as desired.
\end{proof}

The following lemma is used in \cref{eq:implies-main} and \cref{eq:edge-orderstat}.
\begin{lemma}\label{lemma:edge-orderstat}
	Let $\set{\Wok}_{k=1}^{n-1}$ be the order statistics of $n-1$ \iid random variables,
either all uniform $U(0,1)$ or all exponential $\Exp(1)$.
For any $\eps > 0$ and $a= a(n) = \omega(1)$, \whp
	\[ 1-\eps \leq \frac{\Wok}{\E \Wok} \leq 1+\eps \]
	simultaneously for all $k$ in the range $a \leq k \leq n-1$.
\end{lemma}

\begin{proof}

Without loss of generality, we may assume that $a \leq n/10$.

\textbf{Exponential case}.
It is standard that, where $Z_i \sim \Exp(i)$ are independent exponential r.v.s, we may generate the $\Wok$ as
\begin{align}
	\Wok &= \sum_{i=1}^{k} Z_{n-i} . \label{ordersumexp}
\end{align}
Using a superscripted $E$ to highlight the exponential model, $\Wok$ has mean
\begin{align} \label{muX}
 \mu_k = \mu_k\SE \coloneqq \E \Wok
  &=     \sum_{i=1}^k \frac{1}{n-i} = H(n-1)-H(n-k-1)
         \asymp \ln(n)-\ln(n-k) ;
\end{align}
the change by 1 in the logarithms' arguments avoids $\ln(0)$
when $k=n-1$ and remains asymptotically correct.

By~\cref{exp.3},
	\begin{align}
	\Prob( \abs{ \Wok - \mu_k} \geq \eps \mu)
	& \leq
	2 \exp \parens{-\Omega((n-k) \mu_k)}
	.\label{eq:ExpSum}
	\end{align}
By the union bound, it suffices to show that the sum
over $k$ from $a$ to $n-1$ of the RHS of~\cref{eq:ExpSum} is $o(1)$.
We treat the sum in two ranges.
For $k \leq {\frac n2}$,
$(n-k)\mu_k \geq \frac n 2 \cdot \frac k n = \frac k 2$. Thus,
	\begin{align}
	\sum_{k=a}^{{\lfrac n2}} \exp \parens{-\Omega((n-k) \mu_k)}
	\; \leq \; \sum_{k=a}^{{\lfrac n2}} \exp \parens{- \Omega(k)}
	\; \leq \; O(a e^{-\Omega(a)})
	\to 0,
	\end{align}
	since $a = \omega(1)$.
For $k > \frac n 2$, for brevity let $\kbar=n-k$.
Then $\mu_k \asymp \ln n-\ln(\kbar)$ by~\cref{muX}
and
\begin{align}   \notag
\sumkk \exp\parens{ -\Omega((n-k) \mu_k) }
&= \sumkk \exp \parens{ - \kbar \, \Omega(\ln n-\ln(\kbar))}
= \sumkk \parens{\frac{\kbar}{n}}^{\Omega(\kbar)}
\\ & \leq
\parens{\frac1n}^{\Omega(1)} \sumkk
\kbar^{\Omega(1)}
\parens{\frac \kbar n}^{\Omega(\kbar-1)}
 = n^{-\Omega(1)} = o(1) ,
\end{align}
where the explicit inequality factors out the $\kbar=1$ term,
from which, since $\kbar/n \leq 1/2$, the later terms decrease geometrically.
This concludes the exponential case.

\medskip

\textbf{Uniform case}:
Let $U_i \sim U(0,1)$ be \iid uniform random variables and
$W_i \sim \Exp(1)$ \iid exponential random variables.
Because the exponential distribution has CDF $F(x) = 1-\exp(-x)$,
we may couple the two sets of variables as $U_i = F(W_i)$
or equivalently $W_i = f(U_i)$ with $f(x) = F^{-1}(x) = -\ln(1-x)$.
Because $f$ is increasing,
$\Wok = f(\Uok)$.
Now using superscript $U$ to distinguish the uniform model,
the mean is well known to be
\begin{align}\label{muU}
 \mu_k = \mu_k\SU & \coloneqq \E U_{(k)} =  \frac kn
\end{align}

We want to show that with high probability, for all $k$ in the range $a \leq k \leq n-1$,
\[ (1-\eps) \mu_k\SU \leq U_{(k)} \leq (1+\eps) \mu_k\SU  \]
or equivalently,
\[ f\parens{ (1-\eps) \mu_k\SU} \leq \Wok \leq f\parens{ (1+\eps) \mu_k\SU} . \]
From the exponential case already proved,
taking error bound $\eps/2$,
we know that \whp, for all $k$,
\[ (1-\eps/2) \mu_k\SE \leq \Wok \leq (1+\eps/2) \mu_k\SE , \]
so it suffices to show that, for all $k$ (deterministically),
\[ f\parens{ (1-\eps) \mu_k\SU} \leq (1-\eps/2) \mu_k\SE
 \quad \text{ and } \quad
   f\parens{ (1+\eps) \mu_k\SU} \geq (1+\eps/2) \mu_k\SE . \]
This is so. Using~\cref{muU},~\cref{muX}, and convexity of $f$,
\[ f\left((1-\eps) \mu_k\SU \right)
 = f((1-\eps)k/n) \leq (1-\eps) f(k/n)
 = (1-\eps) \ln \left( \frac{n}{n-k} \right)
 \leq (1-\eps/2) \mu_k\SE ; \]
\[ f\left((1+\eps) \mu_k\SU \right)
 = f((1+\eps)k/n) \geq (1+\eps) f(k/n)
 = (1+\eps) \ln \left( \frac{n}{n-k} \right)
 \geq (1+\eps/2) \mu_k\SE . \]
\end{proof}

\section{Upper bound for large \tp{$k$}{k}, sketch}\label{largekUB}
\subsection{Introduction} \label{largekIntro}
To address larger values of $k$ we use a different construction,
generating \st paths of length 4.
A straightforward extension of the previous argument to this construction
would let us get up to $k=n-f(n)$ for an arbitrarily slowly growing function $f$,
but not to $k=n-1$ because it requires $\kp/\eps^2$ edges incident on each of $s$ and $t$
(thus requires that $\kp/\eps^2 \leq n-1$).

Getting all the way to $k=n-1$ requires a couple of additional ideas.
Again, we will introduce an adversary with a cost budget that with high probability exceeds the cost of the first $k$ cheapest paths.
First, we observe that much of the adversary's cost budget
must be spent on edges incident to $s$ and $t$,
leaving less to delete other edges,
thus allowing a smaller structure $R$ to be sufficiently robust.
In particular, the $k$ cheapest paths from $s$ to $t$
must use edges incident on $s$ of total weight at least
$\sum_{i=1}^k \Woi^s$ where
\begin{align} \label{Wkv}
 \Woi^v
\end{align}
is the cost of the $i$th cheapest edge incident to $v$.
(We may omit the superscript when it is either generic or clear from context.)
One technical detail is that,
where $R$ includes the $\kk$ cheapest edges incident to $s$,
we will control $\Wokk-\Wok$ directly,
using results on order statistics from \cref{sec:order-stat},
rather than through a high-probability upper bound on $\Wokk$
and a high-probability lower bound on $\Wok$.
Finally, it is no longer adequate to allow path costs to exceed their
nominal values by an $\eps=\Theta(1)$ factor,
as such large excesses would swell the adversary's budget too quickly,
so we more tightly control the excess cost of each path $P_k$ as a function of $k$
(and $n$, implicitly).

The details later will be clearer if we sketch the argument now,
with most details but without the calculations.
We will argue for $k$ from $n^{4/10}$ to $n-1$.
(We must start with some $k=o(n^{1/2})$ since that is as far as the ``small $k$''
argument extended,
and we need $k=\omega(n^{1/3})$ since below this the new construction's path costs
would exceed the $2k/n$ target.)

\subsection{Structure \tp{$R$}{R}} \label{structR}
\cref{figR2} illustrates the robust structure $R=\Rk$ after adversarial deletion
of root edges, as discussed in \cref{robustness} below.
The construction is based on parameters $\rr0 = \rr0(k)$ and $\eps_k$ to be defined later.
Start with $R$ consisting of just the vertices $s$ and $t$.
Add to $R$ the $\kk$ edges incident on $s$ of lowest cost,
and let $V_s$ be the set of opposite endpoints of these edges.
Do the same for $t$,
generating vertex set $V_t$.
Take $M \coloneqq V(G)\setminus \set{s,t}$ as a collection of ``middle vertices''.

Note that $V_s$, $V_t$, and $M$ may well have vertices in common,
but our analysis will use a subgraph of $R$
where the relevant subsets of these three sets are disjoint,
and it is easier to understand the construction
imagining them to be disjoint.
Add to $R$ each edge $e$
in $M \times V_s$ and $M \times V_t$
that is ``heavy but not too heavy'',
with cost $W(e) \in (\epsk, 2\epsk)$.
This concludes the construction of the structure $R$.
\begin{figure}
  \centering
  \vspace*{2cm}
  \includegraphics[width=0.6\textwidth]{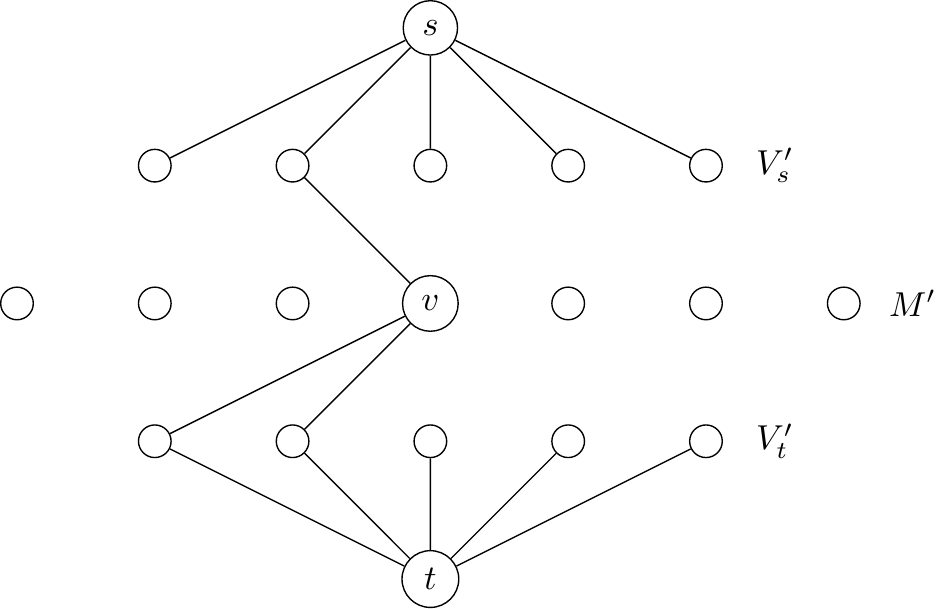}
  \caption{The robust structure $R=\Rk$ after adversarial deletion of $k$ edges
  on $s$, leaving $\rr0$ edges to some vertices $V'_s \subseteq V_s$,
  and likewise for $t$ and $V'_t$.
  The middle vertices are pruned to $M'=M\setminus (V_s'\cup V_t')$,
  and edges from $M'$ to $V_s'$ and $V_t'$ are in $R$
  if they have weight between $\eps_k$ and $2\eps_k$.
  Here, edges from just one representative vertex $v \in M'$ are illustrated.
  }\label{figR2}
\end{figure}

\subsection{Path weights} \label{pathweights}
It is immediate that every \st path in $R$ has cost at most
\begin{align}
	W^s_{(\kk)} +2\epsk+2\epsk+W^t_{(\kk)} . \label{path1}
\end{align}
We will show (in \cref{eq:Wkk0} for uniform and \cref{eq:expWkk0} for exponential) that, subject to the non-occurrence of certain unlikely failure events,
\cref{path1} is at most
\begin{align}
  \Wokp^s+\Wokp^t+7\epsk .  \label{path+}
\end{align}
We will show in \cref{robustness} that,
after deletion of the first $k$ paths, there remains an \st path in $R$
(again subject to non-occurrence of unlikely failure events),
whereupon it follows that
\begin{align}
  X_{k+1} & \leq \Wokp^s+\Wokp^t+7\epsk . \label{XkWok}
\end{align}

\subsection{Adversary} \label{subadv}
We define an adversary who is ``sufficiently strong'' to delete the first $k$ paths.
For $k \leq n^{4/10}$, taking $\eps=0.1$, \cref{ksucceeds} implies that
w.p.\
\begin{align}
  1-O(n^{-1.9}) -O(\exp(-\Omega(n^{4/10})))
   &= 1-O(n^{-1.9})
   = 1-o(1),  \label{Pn0.4}
\end{align}
we have that
\begin{align}
X_k &\leq
X_{n^{4/10}}
 \leq 3 n^{4/10} / n . \label{Xn0.4}
\end{align}
For $k>n^{4/10}$, further assume the absence of the failure events alluded to just above, so that \cref{XkWok} holds.
Then, hiding a sum of the $\ln n/n$ terms of~\cref{Xkbounds} in the $o(\,)$ term below,
\begin{align}
 \sum_{i=1}^{k} X_i
  &=   \sum_{i=1}^{n^{4/10}} X_i + \sum_{i=n^{4/10}+1}^k X_i \notag
  \\ &\leq
  \frac{3{(n^{4/10})}^2}{n} + \sum_{i=n^{4/10}+1}^{k} X_i \notag
  \\ &\leq
	\Ukbaseval + \sum_{i=n^{4/10}+1}^{k} \parens{\Woi^s+\Woi^t+7\eps_{i-1}} \notag
  \\ & =:  U_k .  \label{Ukdef}
\end{align}
Thus, the first $k$ paths' edges have total weight at most $U_k$.

Furthermore, the first $k$ paths' edges incident on $s$ and $t$
are all distinct except, possibly, for the edge \ste.
Therefore, not counting edge \st at all,
the cost of these ``incident'' edges is at least
\begin{align} \label{Ik}
I_k & \coloneqq \sum_{i=1}^{k-1} \parens{ \Woi^s+\Woi^t } .
\end{align}
(In proving \cref{expkbig} we will use a slightly different lower bound
$I_k$ on the weight of the incident edges.)

It follows that the first $k$ paths' ``middle edges'' (edges other than the incident edges)
cost at most $U_k-I_k$.
We will explicitly define a budget $B_k$ satisfying
\begin{align}\label{budget}
  B_k & \geq U_k-I_k .
\end{align}

We will allow the adversary to delete
any $k$ edges in $G$ incident on each of $s$ and $t$,
possibly including the edge \st
(enough to let it delete the incident edges of the first $k$ paths),
and to delete any other edges in $G$ of total cost at most $B_k$
(enough to let it delete the middle edges of the first $k$ paths).
Thus, the adversary is sufficiently strong to delete the first $k$ paths.

The adversary's allowable deletions in $G$ mean that also in $R$
it deletes at most $k$ edges incident on each of $s$ and $t$,
and middle edges of total cost at most $B_k$.

\subsection{Budgets \tp{$B_k$}{B\_k}} \label{budgets}
The budgets $B_k$ will be defined explicitly in the details.
For the model with uniformly distributed edge weights
we will do so in two ranges of $k$,
corresponding to \cref{kmedium,kbig},
and likewise in the model with exponentially distributed edge weights,
corresponding to \cref{expkmedium,expkbig}.
For \cref{kbig,expkbig} we will establish \cref{budget} directly.

For \cref{kmedium,expkmedium} we will establish \cref{budget}
by the following reasoning;
we will only need to check \cref{Bksuff}, \cref{Bkbase}, and \cref{epsfit} below.
We will show that the budgets satisfy
\begin{align}\label{Bksuff}
  B_\kp
   &\geq B_k + 8\epsk .
\end{align}
(Roughly speaking, given $B_k$ we will set $\eps_k$ as small as possible
while keeping $\Rk$ robust to the adversary with budget $B_k$.
Then, we will set $B_\kp$ as small as possible,
namely by taking equality in \cref{Bksuff}.
Behind the scenes, we derive $B_k$ by solving the differential-equation
equivalent of \cref{Bksuff} satisfied with equality.)

We will show that \cref{budget} is satisfied in the base case,
by showing that
\begin{align} \label{Bkbase}
 B_k &\geq U_k \quad \text{for $k=n^{4/10}$} .
\end{align}
Then, \cref{budget} is established for all $k$ by induction on $k$:
\begin{align}
  U_\kp-I_\kp
   &= (U_\kp-U_k)-(I_\kp-I_k)+ (U_k-I_k) \notag
  \intertext{which by \cref{Ukdef}, \cref{Ik}, and the inductive hypothesis \cref{budget} is}
   & \leq (\Wo \kp^s+\Wo \kp^t+7\epsk)-(\Wok^s+\Wok^t) + B_k \notag
   \\ & \leq B_k + 8\epsk \eqnote{see below} \label{epsfit}
   \\ & \leq B_\kp  \eqnote{by \cref{Bksuff}} \label{Bkp} .
\end{align}
To justify \cref{epsfit} it suffices to show that $\Wokp-\Wok$
is at most $0.1 \eps_k$,
and we do so
in \cref{epsfitpf} for the uniform case and
in \cref{expepsfitpf} for the exponential case.
In both cases, $\rr0=\omega(1)$,
and $\Wo\kk-\Wokp=O(\eps_k)$
(as used in going from \cref{path1} to \cref{path+}),
making this conclusion unsurprising.%
\footnote{In proving \cref{kbig,expkbig} we will set $\rr0=1$,
so this reasoning does not apply.
Indeed, in \cref{expkbig} (the large-$k$ exponential case)
\cref{epsfit} would be false
--- $\Wokp-\Wok$ can be much larger than $\eps_k$ ---
but (to reiterate) it is not needed there,
as we establish \cref{budget} directly.
}

\subsection{Robustness of \tp{$R$}{R}} \label{robustness}
We wish to make $R$ robust against the adversary,
so that after the deletions just described,
$R$ should retain an \st path \whp,
so that \cref{XkWok} holds and $X_\kp$ is small.
It will suffice to show that, to delete all \st paths in $R$,
\begin{align}\label{robustblurb}
  \parbox{0.8 \textwidth}{\emph{after deletion of $k$ edges incident on each of $s$ and $t$,
  an adversary would still have to delete middle edges of total cost more than $B_k$,}}
\end{align}
and thus it is powerless to do so.

Obtaining this robustness requires choosing $\eps_k$ sufficiently large in the construction.
With reference to \cref{figR2},
on deletion of any $k$ edges on each of $s$ and $t$,
the level-1 sets are in effect pruned to $V'_s$ and $V'_t$,
each of cardinality $\rr0$.
Should $V'_s$ and $V'_t$ have vertices in common,
or if $t \in V'_s$ or $s \in V'_t$,
then there is an \st path.
So, assume that $V'_s$ and $V'_t$ are disjoint and do not contain $s$ nor $t$.
Consider only middle vertices $M' \subseteq M$ not appearing in $V'_s$ nor $V'_t$,
i.e., $M'=M \setminus \set{V'_s \cup V'_t}$.
We will have $\rr0=o(n)$, so
$\card{M'} = n-2-2\rr0 > 0.99n$.
Note that edges in $M' \times V'_s$, $M' \times V'_t$,
$\set s \times V'_s$, and $\set t \times V'_t$
are all distinct.

Consider a choice of the $k$ deletions on $s$ and $t$ to be fixed in advance.
(We will eventually take a union bound over all such choices.)
The weights of edges in $M' \times V'_s$ and $M' \times V'_t$
have not even been observed yet,
so each has (unconditioned) $U(0,1)$ distribution,
all are independent (by distinctness of the edges),
and thus each such edge is included in $R$
with probability $\epsk$, independently.

A vertex $v \in M'$ is connected to $V'_s$
by
\begin{align} \label{Zvs}
Z_v^s & \sim \Bi(\rr0,\epsk)
\end{align}
edges, with mean
\begin{align}\label{kmedlambda}
 \lambda \coloneqq \E Z_v^s = \rr0 \epsk .
\end{align}
Define $Z_v^t$ symmetrically, and note that $Z_v^s$ and $Z_v^t$ are \iid.
Intuitively, if $\lambda$ is small,
$Z_v^s$ is usually 0, is 1 with probability about $\lambda$,
and rarely any larger value.
So, the probability that $v$ is connected to both $V'_s$ and $V'_t$
is about $\lambda^2$,
in which case to destroy \st paths through $v$
the adversary must delete an edge of cost at least $\eps_k$.
So, to delete all \st paths, over the nearly $n$ vertices in $M'$
the adversary would have to delete edges of expected total weight at least
\begin{align}
\eps_k \, n \, \lambda^2 \label{kmedtotalweight}
 .
\end{align}
We will choose $\eps_k$ so that
\begin{align} \label{epsk}
 \eps_k \, n \, \lambda^2 > B_k ,
\end{align}
which hopefully will ensure (see \cref{failprob})
that a path must remain (i.e., that $R$ is robust).

Let us give a back-of-the-envelope calculation.
In the uniform case we expect $\Wok$ to be about $k/n$,
so letting $\rr0= \eps_k n$ means that $\Wokk-\Wok$ will be about $\eps_k$,
justifying \cref{path+}.
Then \cref{kmedlambda} gives $\la=\epsk^2 n$,
so \cref{epsk} indicates that we need to take
$\epsk^5 n^3 > B_k$.
As noted after \cref{Bksuff}, roughly speaking,
we obtain $B_k$ and $\epsk$
by solving this and \cref{Bksuff} with equality as a system of differential equations.

\begin{remark} \label{failprob}
This intuitive argument proves to be essentially sound,
but to make it rigorous will take some work.
Chiefly, $\Pr(Z_v^s > 0)$ is of course not exactly $\E Z_v^s = \lambda$
even when $\lambda$ is small, and we will also have to consider the case
when $\lambda$ is large.
Also, where the intuition is based on expectations,
we must calculate the probability of the ``failure'' event that
all paths can be deleted at a cost less than $B_k$.
Finally, we must take the union bound of this failure event over all choices
of root edges at $s$ and $t$
(but, as in the small-$k$ case, this turns out to change nothing).
\end{remark}

\section{Upper bound for large \tp{$k$}{k}, uniform model} \label{unifUB}
In this section
we fill in the details of the steps
from \cref{largekUB}
and show that they conclude the proof of the upper bound in \cref{Tmain}.
Specifically, to control the \emph{path weights}
(these emphasised keywords match section titles)
we must show that
\cref{path1} is at most \cref{path+}.
For the \emph{adversary} we need only show \cref{budget};
as noted earlier, for large $k$ (\cref{kbig}) we will do this directly,
while for medium $k$ (\cref{kmedium}) we will argue
that the \emph{budgets} $B_k$ satisfy \cref{Bkbase} and \cref{epsfit}.
And for \emph{robustness} we will prove that the
probability of failure is small
(i.e., it is unlikely that the adversary can destroy all \st paths in $\Rk$).

\subsection{Claims, and implications for \cref{Tmain}}\label{uniclaims}

We first state the two precise claims we make for large $k$, in two ranges.
We use symbolic constants $\CB$, $\Ce$, $C'_B$, and $C'_\eps$
in the claims and the proofs,
as it makes the calculations clearer.
Whenever we encounter an inequality that the constants must satisfy,
we will highlight with a parenthetical ``check'' that they do so.

\begin{claim}\label{kmedium}
For $k \in [n^{4/10}, n-14\sqrt n \,]$,
let $\Bk = \parens{\CB k n^{-3/5}+{(2n^{-1/5})}^{4/5}}^{5/4}$
and $\epsk =\Ce n^{-3/5} {\Bk}^{1/5}$,
with $\CB=32$ and $\Ce=5$.
Then, asymptotically almost surely,
simultaneously for all $k$ in this range,
\begin{equation}\label{eq:XkWk}
	X_\kp \leq \Wokp^s + \Wokp^t + 8 \epsk.
\end{equation}
\end{claim}

\noindent\textbf{Remark:}
In proving \cref{kmedium} we will set
\begin{align}\label{k0med}
  \rr0 & \coloneqq \eps_k n .
\end{align}
From the definitions of $B_k$ and $\epsk$ in \cref{kmedium}, both are increasing in $k$,
and we will make frequent use of the following inequalities. For $n$ sufficiently large,
\newline
\noindent
\begin{minipage}[t]{.5\textwidth}
\begin{align}
	\Bk &= \Theta \parens{ k^{5/4} n^{-3/4} + n^{-{1/5}}} \label{eq:Bk} \\
	B_k &\leq B_n \leq 1.01 \CB^{5/4} n^{1/2} \label{eq:Bkupper}\\
	B_k &\geq B_{n^{4/10}} \geq 2n^{-{1/5}} \label{eq:Bklower}
\end{align}
\end{minipage}%
\begin{minipage}[t]{.5\textwidth}
\begin{align}
    \epsk &= \Theta \parens{ k^{1/4} n^{-3/4}+ n^{-{16/25}} } \label{eq:ek} \\
	\epsk &\leq \eps_n \leq 1.01 \Ce \CB^{1/4} n^{-1/2} \label{eq:ekupper} \\
	\epsk &\geq \eps_{n^{4/10}} \geq 1.14 \Ce n^{-{16/25}} \label{eq:eklower}
    .
\end{align}
\end{minipage}%

\begin{claim}\label{kbig}
For $k \in ( n-14\sqrt n, n-2]$,
let
\begin{align}
 \Bp&=\CB' \sqrt n \quad \text{and} \quad \epsp = C'_\eps n^{-1/6} \label{kbigparams}
\end{align}
with $C'_B=78$ and $C'_\eps=5$.
Then, asymptotically almost surely,
simultaneously for all $k$ in this range,
\[ X_\kp \leq \Wokp^s + \Wokp^t + 8 \epsp. \]
\end{claim}

\noindent\textbf{Remark:} In proving \cref{kbig} we will set $\rr0 \coloneqq 1$.
Note that here $B_k$ and $\eps_k$ are constants independent of $k$,
but we retain the subscript for consistency with the notation of \cref{largekIntro}.

We will prove the two claims shortly.

\medskip

\begin{proof}[Proof of the upper bound of \cref{Xkbounds} in \cref{Tmain}]
Given $\eps > 0$ from \cref{Tmain}, apply \cref{lemma:edge-orderstat} to the order statistics $\Wok^s$ and $\Wok^t$ with $\eps$ in the lemma as our $\eps/2$ and $a=n^{4/10}$. Then by \cref{kmedium} \whp, simultaneously for all
$k \in [n^{4/10}+1, n-14\sqrt n]$,
\begin{equation}\label{eq:implies-main}
	X_\k \leq \Wok^s + \Wok^t + 8\eps_\km
     \leq (1+\eps/2) 2k/n + 8\eps_\km
     \leq (1+\eps) (2k/n + \ln n/n);
\end{equation}
the key point is that $\eps_{\km} \leq \epsk = o(k/n)$, which follows from~$\cref{eq:ek}$.
Specifically, by \cref{eq:ek}, $\epsk/(k/n) = O(k^{-3/4} n^{1/4}+k^{-1}n^{9/25})$,
which by $k \geq n^{4/10}$ is $O(n^{-0.3}n^{0.25}+n^{-4/10}n^{0.36})=o(1)$.

Likewise, by \cref{kbig},
inequality~\cref{eq:implies-main} holds
\whp simultaneously for all $k \in [n-14\sqrt n, n-2]$.
Again, we need only show that $\epsp=o(k/n)$,
which holds because here $k/n=\Theta(1)$ while by definition $\epsp = o(1)$.
\end{proof}

We now prove the two claims, by filling in the details
for \cref{structR,robustness}.

\subsection{Structure \tp{$R$}{R}}

With reference to \cref{structR},
all that we need to confirm is that $\kk \leq n-1$.
For \cref{kmedium},
by hypothesis $k \leq n-14 \sqrt n$,
and provided that $1.01 \Ce \CB^{1/4} \leq 13$
(check), by~\cref{eq:ekupper}
$\epsk \leq 13 n^{-1/2}$,
whereupon $\rr0 = \epsk n \leq 13 \sqrt n$.
For \cref{kbig}, with $\rr0=1$,
$\kk \leq n-1$ is immediate.

\subsection{Path weights} \label{pathweightsdetail}
With reference to \cref{pathweights},
we establish that the bound \cref{path+} holds \whp simultaneously for all $k \geq n^{4/10}$.
With $\Wok$ representing the cost of the $k$th cheapest edge incident on some fixed vertex (which we will take to be $s$ and then $t$ in turn), it suffices to show that
\begin{align}\label{eq:Wkk0}
\Wokk &\leq \Wokp + 1.1\epsk
\end{align}
holds with high probability for all $k \geq n^{4/10}$.

For \cref{kbig}, with $\rr0=1$, \cref{eq:Wkk0} is immediate.
For \cref{kmedium}, with $\rr0=\eps_k n$,
generate the variables $\Wok$ by placing $n-1$ points uniformly at random on the unit interval $I$, associating $\Wok$ with the $k$th smallest point.
It suffices to show that, \whp,
each interval $(\Wokp, \Wokp+1.1\epsk)$ contains at least $\rr0$ points.
For all $k \in [n^{4/10}, n-14\sqrt n-1\,]$,
\cref{kmedium} has $\epsk \geq n^{-0.99}$ by~\cref{eq:eklower},
so \cref{intervals} shows that
w.p.\ $1-\exp(-\Omega(n^{0.01}))$,
every interval of length $\geq 1.1\epsk$ in $[0,1]$
contains at least $\rr0 \eqdef \epsk n$ points,
and in particular this holds for all the intervals $(\Wokp, \Wokp+1.1\epsk)$.

We assume henceforth that the graph $G$ is ``good'' in the sense that
\cref{eq:Wkk0} holds for all $k \geq n^{4/10}$ for vertices $s$ and $t$,
and that
for all $k \leq n^{4/10}$ we
have the upper bounds on $X_k$ from~\cref{Xkbounds},
as proved to hold w.h.p.\ in \cref{sec:k-small}.

\subsection{Adversary}
With reference to \cref{subadv}, we need only verify \cref{budget},
and this will be done in the next subsection.

\subsection{Budgets \tp{$B_k$}{B\_k}}
With reference to \cref{budgets},
we first establish \cref{epsfit}.
This follows from
\begin{align}\label{epsfitpf}
  \Wo \kp^s-\Wok^s & \leq 0.1 \epsk .
\end{align}
The reasoning for this is the same as for \cref{eq:Wkk0}:
each interval of length $0.1 \epsk$ contains at least one point.
The parameters are trivial to check.

Next, we show that the parameters of \cref{kmedium} satisfy \cref{budget},
for which as argued in \cref{budgets} it suffices to show that they satisfy
\cref{Bksuff} and \cref{Bkbase}.
We start with \cref{Bkbase}, the base case.
Here $k=n^{4/10}$,
$B_k \geq \Ukbaseval$ from~\cref{eq:Bklower},
and $U_k = \Ukbaseval$ from \cref{Ukdef}, establishing \cref{Bkbase}.

To establish \cref{Bksuff}, first
note that $\frac{\partial}{\partial k} \Bk = \tfrac 54 \CB {\Bk}^{1/5}n^{-3/5}$ is an increasing function.
Then,
\[ B_\kp-\Bk \geq \frac{\partial}{\partial k} \Bk = \frac 54 \CB {\Bk}^{1/5}n^{-3/5} = \frac 54 \frac {\CB}{\Ce} \epsk \geq 8\epsk, \]
since $\CB \geq \tfrac{8 \cdot 4}{5} \Ce$ (check).

We now establish \cref{budget} for the parameters of \cref{kbig}.
With $\kstar= \floor{n-14\sqrt n}$,
the point where \cref{kmedium} ends and just before \cref{kbig} begins,
the previous case showed that $B_\kstar \geq U_\kstar - I_\kstar$,
and by~\cref{eq:Bkupper} $B_\kstar \leq 77 \sqrt n$.
Then, for $k$ from $\kstar+1$ to $n-2$,
\begin{align}
U_k - I_k
  &= (U_\kstar - I_\kstar) + [(U_k-U_\kstar) - (I_k - I_\kstar)]
  \notag \\&\leq B_\kstar + \left[
    \sum_{i=\kstar+1}^k (\Woi^s + \Woi^t + 7 \eps_{i-1}) -
    \sum_{i=\kstar}^{k-1} (\Woi^s + \Woi^t)\right]
    \eqnote{see \cref{Ukdef} and \cref{Ik}}
  \notag \\&\leq B_\kstar + \sum_{i=\kstar+1}^{n-2} 7 \eps_{i-1} + (\Wok^s + \Wok^t - \Wo\kstar^s - \Wo\kstar^t)
  \notag \\&\leq 77 \sqrt n + (14\,\sqrt n) \cdot 7 C'_\eps n^{-1/6} + 2
    \eqnote{see \cref{kbigparams}}
  \notag \\& \leq 78 \sqrt n
  \notag \\& \leq \Bp \eqnote{see \cref{kbigparams}}, \label{budgetsClaim2}
\end{align}
using that $\CB' \geq 78$ (check).

\subsection{Minimum of two binomial variables}
Before addressing robustness of the structure $R$,
we require a lemma (\cref{lemma:BinMin}) on the minimum $Z$
of two \iid binomial $\Bi(n,p)$ random variables.
There is a genuine difference in the cases when the common mean $\lambda=np$ is large or small:
if $\lambda$ is large then $Z$ is likely to be
close to $\la$, making $\E Z = \Theta(\la)$;
if $\lambda$ is small then $Z$ will most often be 0,
occasionally 1 (with probability about $\lambda^2$), and rarely anything larger,
making $\E Z = \Theta(\la^2)$.
The lemma relies on the following property of the median of a binomial random variable.
(A weaker form of~\cref{Lbigla}
and thus of \cref{lemma:BinMin}
can be obtained from \cref{lemma:BinDev} in lieu of using the median.)
\begin{theorem}[Hamza~{\cite[Theorem 2]{Hamza1999}}]
A binomial random variable $X$ has median satisfying $|\Med(X)-\E X| \leq \ln 2$.
\end{theorem}
\noindent
In this discrete setting $\Med(X)$ is not unique: it can be any value $m$ for which
$\Pr(X \leq m) \geq 1/2$ and $\Pr(X \geq m) \geq 1/2$.
\cite{Hamza1999} defines it uniquely as the smallest integer $m$ such that
$\Pr(X \leq m) >1/2$;
as desired, this gives
$\Pr(X \geq \Med(X)) = 1-\Pr(X \leq \Med(X)-1) \geq 1-1/2=1/2$.
(For other results on the binomial median see
Kaas and Buhrman~\cite{Kaas1980}, in particular, Corollary~1.
Stronger results for the Poisson distribution are given by
Choi~\cite{Choi1994},
proving a conjecture of Chen and Rubin,
and by Adell and Jodr\'a~\cite{Adell2005}.)

\begin{lemma}\label{lemma:BinMin}
	Let $Z_1, Z_2$ be \iid $\Bi(n, p)$ random variables, $Z \coloneqq \min(Z_1, Z_2)$ and $\lambda \coloneqq \E Z_1 = np$.
	\begin{enumerate}[(1)]
		\item If $\lambda \geq 2$, then
		\begin{align}
			\Prob(Z \geq 0.65 \lambda) &> 1/4. \label{Lbigla}
		\end{align}
		\item If $\lambda \leq 2$, then
		\begin{align}
			\Prob(Z \geq 1) &> 0.18 \la^2 . \label{Lsmallla}
		\end{align}
	\end{enumerate}
\end{lemma}
\begin{proof}
In the first case,
\[ \Med(Z_1) \geq \la-\ln 2 = \frac{\la-\ln 2}\la \la
 \geq \frac{2-\ln 2}2 \la \geq 0.65 \la , \]
so
$\Prob \parens{Z_1 \geq 0.65 \la}
 \geq \Prob \parens{Z_1 \geq \Med(Z_1)} \geq 1/2$.
The same holds of course for $Z_2$, and the result follows by independence.

In the second case
we again use independence, and here
\begin{equation*}
\Prob \parens{Z_1 \geq 1} = 1-{(1-p)}^n \geq 1-\exp(-\lambda)
 = \frac{1-\exp(-\lambda)}{\lambda} \cdot \lambda \geq 0.43\lambda .
\end{equation*}
The last inequality comes from minimising $\frac{1-\exp(-x)}{x}$ over $0 \leq x \leq 2$;
the function is decreasing so the minimum is at $x=2$.
\end{proof}

\subsection{Robustness in \cref{kmedium}} \label{kmedrobust}
With reference to \cref{robustness},
let us complete the robustness argument for \cref{kmedium},
showing that \cref{robustblurb} holds with high probability.
Here we have taken $\rr0 = \eps_k n$,
so that the number of edges from a middle vertex to $V'_S$
(see \cref{Zvs})
is $Z^s_v  \sim \Bi(\eps_k n, \eps_k)$,
with mean $\la=\rr0 \eps_k = \eps_k^2 n$ (see \cref{kmedlambda}).

Recall that if $\la$ is small we expect (see \cref{kmedtotalweight})
that to destroy all paths the adversary will have to delete
edges of total weight at least
$\eps_k \, n \, \lambda^2 = \epsk^5 n^3$,
which will exceed $B_k$.
And, if $\la$ is large, then each $Z_v$ will have expectation close to $\la=\epsk^2 n$, for a total cost $\epsk n$ times larger, namely $\epsk^3 n^2$,
and again this exceeds $B_k$.
We now replace these rough calculations with detailed probabilistic ones,
applying \cref{lemma:BinMin} to $Z_v$ in the
two cases of $\la$ small and large.

For the adversary to delete all \st paths via $v$, he must delete at least
\[ Z_v\coloneqq \min(Z_v^s, Z_v^t) \]
edges, and to destroy all paths he must delete at least
\[ N\coloneqq \sum_{v \in M'} Z_v \]
edges.
As described in \cref{robustness},
we imagine a fixed deletion of $k$ edges on each of $s$ and $t$
giving neighbour sets $V'_s$ and $V'_t$
and a set $M'$ of middle vertices;
we will eventually take a union bound over all such choices.

\medskip \noindent \tmbf{If $\lambda \geq 2$}, then
by \cref{lemma:BinMin},
for each $v \in M'$, $\Pr(Z^s_v \geq 0.65 \la) \geq 1/4$.
Thus, $N$ stochastically dominates $0.65\lambda \cdot \Bi(0.99n, 1/4)$,
with expectation $> 0.1608 \la n$.
We shall consider it a \emph{failure} if $N \leq 0.16 \lambda n$.
Assuming success,
since each edge costs at least $\epsk$ to delete,
it costs at least $0.16 \epsk \la n = 0.16 \epsk^3 n^2$ to delete them all.
This exceeds $B_k$:
\begin{align*}
	\frac{0.16 \cdot \epsk^3 n^2}{B_k}
	&= 0.16 \cdot \Ce^3 n^{-9/5} B_k^{-2/5} n^2 \eqnote{by definition of $\eps_k$}
    \\& \geq 0.15 \cdot \Ce^3 \CB^{-1/2} n^{1/5} n^{-1/5}
        \eqnote{by~\cref{eq:Bkupper}}
    \\ &> 1,
\end{align*}
using that $0.15 \cdot \Ce^3 \CB^{-1/2}>1$ (check).

Failure means that $N/(0.65 \la) \sim \Bi(0.99n, 1/4) \leq (0.16/0.65)n$.
Noting that $0.99 \cdot 1/4 > 0.16/0.65$,
by \cref{lemma:BinDev}, the probability of failure is
$\exp(-\Omega(n))$.
By the union bound,
the total of the failure probabilities,
over all rounds (values of $k$) and all adversary choices of the $k$ root edges at $s$ and $t$, is small:
\begin{align} \label{case1failure}
\sum_k {{\binom{\kk}{\rr0}}^2} & \cdot \exp(-\Omega(n))
 \\& \leq \sum_k \, {(n^{\rr0})}^2 \exp(-\Omega(n))
    \notag
 \\ &= \sum_k \exp\parens{2\epsk n \ln n-\Omega(n)}
    \quad\text{(by $\rr0=\epsk n$)}  \notag
 \\&\leq n \exp(-\Omega(n))
  \quad\text{(using $\epsk n = O(n^{1/2})$ from~\cref{eq:ekupper})} \notag
 \\& = o(1) . \notag
\end{align}

\medskip \noindent \tmbf{If $\lambda < 2$},
then by \cref{lemma:BinMin} $N$ stochastically dominates $\Bi(0.99n, 0.18\lambda^2)$,
with expectation $> 0.175 \la^2 n$.
We shall consider it a \emph{failure}
if $N \leq 0.17 \lambda^2 n = 0.17 \epsk^4 n^3$.
Each edge costs at least $\epsk$ to delete.
Assuming success, it thus costs at least $0.17 \epsk^5 n^3$ to delete them all, which exceeds $B_k$:
\begin{align*}
	\frac{0.17 \epsk^5 n^3}{B_k}
 &= 0.17 \Ce^5  \eqnote{by definition of $\eps_k$}
 \\ &> 1,
\end{align*}
using that $0.17 \Ce^5>1$ (check).

By \cref{lemma:BinDev}, the probability of {failure} is
\begin{equation}\label{eq:N2}
	\Prob \parens{ N \leq 0.17 \epsk^4 n^3} = \exp(-\Omega(\epsk^4 n^3)).
\end{equation}
Over all rounds (values of $k$) and adversary choices of edges incident to $s$ and $t$,
the total failure probability is at most
\begin{align}
\sum_k {\binom{\kk}{\rr0}}^2 & \cdot \Prob \parens{N < 0.17 \epsk^4 n^2}  \notag
 \\ &\leq \sum_k \exp\parens{2\epsk n \ln n- \exp(-\Omega(\epsk^4 n^3))}  \notag
 \\&\leq n \exp(-\Omega(\epsk^4 n^3)) , \notag
\intertext{because $\epsk n \ln n$ is dominated by $\epsk^4 n^3$:
the latter is larger by a factor $\epsk^3 n^2/\ln n$,
which by~\cref{eq:eklower} is
$\Omega(n^{-48/25}n^2/\ln n)=\Omega(n^{2/25}/\ln n)=\omega(1)$.
Continuing, this is}
 &\leq n \exp(-\Omega(n^{11/25}))  \eqnote{invoking~\cref{eq:eklower} again} \notag
 \\& = o(1) .  \label{case2failure}
\end{align}

\bigskip

\subsection{Robustness in \cref{kbig}} \label{kbigrobust}

Again, our aim is to establish robustness of $R$ by
showing that \cref{robustblurb} holds with high probability,
and the argument is similar to but simpler than that of \cref{kmedrobust}.

Since $\rr0=1$, both $V'_s$ and $V'_t$ have size 1.
For a vertex $v \in M'$, let $Z_v$ be the number of paths from $V'_s$ to $V'_t$ via $v$. There is only one such possible path, hence
\[ Z_v \sim \Bern \parens{\epsp^2}. \]
To destroy all \st paths the adversary must delete at least
\[ N\coloneqq \sum_{v \in M'} Z_v \]
edges. $N$ stochastically dominates $\Bi(0.99n, \epsp^2)$,
which has expectation $0.99 \epsp^2 n$.
We declare the event $N \leq 0.98 \epsp^2 n$ a \emph{failure}.
Assuming success, destroying all \st paths would cost at least
$\epsp N \geq 0.98 \epsp^3 n$.
This exceeds $\Bp$, since
\begin{equation*}
	\frac{0.98 \epsp^3 n}{B_k} = \frac{0.98 {C'_\eps}^3}{C'_B},
\end{equation*}
and $0.98 {C'_\eps}^3 > C'_B$ (check).

The probability of failure is
\begin{equation}\label{eq:N3}
\Prob \parens{ N \leq 0.98 \epsp^2 n}
    = \exp(-\Omega(\epsp^2 n))
    = \exp(-\Omega(n^{2/3})).
\end{equation}
Over all rounds and adversary choices,
using that $\binom{\kk}{\rr0} = \binom{k+1}1 \leq n$,
the total failure probability is at most
\begin{align}
 \sum_k {\binom{\kk}{\rr0}}^2 & \cdot \Prob(N \leq 0.98 \eps^2 n) \notag
  \\ &\leq
	(14\sqrt n) \, n^2 \,
 \exp (-\Omega(n^{2/3}))  \eqnote{by~\cref{eq:N3}} \label{eq:unifailure3}
      \\& = o(1) . \notag
\end{align}

\section{Lower bound}\label{lowerbound}
In this section, we establish the lower bound in \cref{Xkbounds}
of \cref{Tmain}.
\cref{LBsmallk} establishes the lower bound on $X_k$
directly for $k \leq \sqrt{\ln n}$.
Values $k \geq \sqrt{\ln n}$ are treated in the subsequent parts.
In \cref{LBrunning},
\cref{lemma:S_k-lower} establishes a lower bound
on the running totals $S_k$,
\begin{align}\label{Skdef}
S_k \coloneqq \sum_{i=1}^{k} X_i .
\end{align}
In \cref{LBlargek},
\cref{lemma:X_k-lower} obtains a lower bound on $X_k$
using \cref{lemma:S_k-lower}'s lower bound on $S_k$,
the previously established upper bound on $X_k$ from \cref{Tmain},
and the monotonicity of $X_k$.

\subsection{Lower bound for small \tp{$k$}{k}} \label{LBsmallk}
We begin with $k \leq \sqrt{\ln n}$.
For any fixed $\eps>0$, we know from~\cite{Janson123} that \whp
\begin{align}\label{X1lower}
  X_1 &> (1-\eps/2)\frac{\ln n}{n}  .
\end{align}
Assuming that~\cref{X1lower} holds,
it follows immediately, and deterministically, that for all $k \leq \sqrt{\ln n}$,
\begin{align}\label{Xklower1}
X_k &\geq X_1 \geq (1-\eps/2)\frac{\ln n}{n} \geq (1-\eps)\frac{2k+\ln n}{n} .
\end{align}
The first inequality holds because the sequence $X_k$ is monotone increasing,
the next by assumption on $X_1$,
the next by $k = o(\ln n)$.

\subsection{Lower bound on the running totals} \label{LBrunning}

\begin{lemma}\label{lemma:S_k-lower}
For any $\eps>0$, \whp, simultaneously for every
$k \leq \Khigher$,
\begin{align}\label{Sklower}
 S_k &\geq \SkLower.
\end{align}
\end{lemma}

\begin{proof}
	Write $\Woi^s$ and $\Woi^t$ for the order statistics of edge weights out of $s$ and $t$, respectively.
By \cref{lemma:edge-orderstat}, \whp,
	\begin{equation}\label{eq:edge-orderstat}
	\Wok^s, \Wok^t \in \left[\left(1-\eps/2 \right) \frac k n,
                       \left(1+\eps/2 \right) \frac k n \right]
       \quad \text{ for all } k \geq \Klower ,
	\end{equation}
and we will assume throughout the proof that~\cref{eq:edge-orderstat} holds.

We prove the assertion in two ranges of $k$.
	
\medskip\noindent\tmbf{For $\KrangeHigh \leq k \leq n-1$,}
the $k$ paths must use at least $k-1$ edges on each of $s$ and $t$, all distinct
($k$ edges each, ignoring the edge $\set{s,t}$ if it is used).
Then, using~\cref{eq:edge-orderstat}, we get that \whp, for all $k$ in the range,
\begin{align}
S_k &\geq \sum_{i=1}^{k-1} \left( \Woi^s + \Woi^t \right)
  \geq \sum_{i=\Klower}^{k-1} (1-\eps /2) \frac{2 i}n  \notag
  \\&
  = (1-\eps/2) \parens{
    \sum_{i=1}^{k} {\frac{2i+\ln n}n}
    - \sum_{i=1}^{k} \frac{\ln n}n
    - \sum_{i=1}^{\Klower-1} {\frac{2i}n}
    - {\frac{2k}n}
   }
    \label{SbigK1}
   \\& \geq (1-o(1)) (1-\eps/2) \sum_{i=1}^{k} {\frac{2i+\ln n}n} \label{SbigK2}
     \eqnote{see below}
  \\&\geq \SkLower  \label{SbigK}
.
\end{align}
To justify \cref{SbigK2} it suffices to show that
the first sum in \cref{SbigK1} is of strictly larger order than the other terms.
The first sum is at least
$\sum_{i=k/2}^{k} 2i/n = \Omega(k^2/n)$,
which since $k \geq \ln^{11/10}n$ is also
$\Omega(k \ln^{11/10}n/n)$ and $\Omega(\ln^{22/10}n/n)$;
we will use all three formulations.
The second term is of order $O(k \ln n/n)$,
negligible compared with the middle formulation.
The third term is
$O(\ln^{2/3}n/n)$,
negligible compared with the last formulation.
And the fourth term, of order $O(k/n)$, is negligible
compared with the first formulation.

\medskip\noindent\tmbf{For $1 \leq k \leq \KrangeHigh$,}
let $\delta=\eps/3$
and let $G'=G-s-t$.
Let $N_s$ and $N_t$ be the endpoints of the cheapest $\nto$ edges out of $s$ and $t$ respectively.
Note that these sets are independent of the edge weights of $G'$.

If any path $P_i$, $i \leq k$,
uses a root edge (edge incident on $s$ or $t$) \emph{not}
among the $\ln^3 n$ cheapest edges of $s$ or $t$,
then by~\cref{eq:edge-orderstat}
this edge costs at least $(1-\eps) \ln^3 n / n$,
thus $S_k \geq (1-\eps) \ln^3 n / n$.
Then \cref{Sklower} follows because this is larger than
the RHS of~\cref{Sklower}, namely $\Theta((k^2+k\ln n)/n) = O(\ln^{11/5} n/n)$
for this range of $k$.
Thus we may assume that for all $i \leq k$,
each path $P_i$ goes via some $s' \in N_s, \; t' \in N_t$.

For $s' \in N_s, \; t' \in N_t$, define $A(s',t')$ to be the event that
$t'$ is one of the ${(n-2)}^{1-\delta}$ nearest vertices of $s'$, by cost, in $G'$.
Clearly, for each pair $s', t'$, $\Pr(A(s',t')) = {(n-2)}^{-\delta}$.
Let $A$ be the union of these events, i.e., the event that any such pair has this property.
By the union bound,
 	\[ \Prob(A) \leq {\left(\nto \right)}^2 {(n-2)}^{-\delta} = o(1) . \]
We assume henceforth that $A$ does not hold:
the $\ln^3n$ cheapest root edges at $s$ and $t$ do not happen to sample
any ``nearest''
pairs in $G'$.
 	
By assumption that $A$ does not hold,
in the \emph{exponential} model (where each edge is i.i.d.\ $\Exp(1)$) for $G'$, for each $s' \in N_s, \; t' \in N_t$,
the distance $d(s',t')$ stochastically dominates $Y \sim \sum_{i=1}^{n^{1-\delta}} \Exp(i(n-2-i))$ by \cref{treeX}.
We have $\E Y = (1+o(1))(1-\delta) \ln n / n$ by
~\cref{treeDia} (just adjusting its last equation where the value of $d$ is substituted in).
Applying \cref{exptail}'s \cref{exp.3}
with $\mu=\E Y$ as above, $\astar=n-3$, and $\lambda=1-\delta$,
that in the exponential model $G'$,
\begin{align*}
\Prob & \left(d_{G'}(s',t')
        \leq (1-\delta) \frac{(1+o(1))(1-\delta) \ln n}{n} \right)
     \; \leq \; \exp\parens{ -\Theta(n \cdot \ln n/n \cdot \delta^2) }
     \;    = \; n^{-\Theta(1)}.
\end{align*}
Since ${(1+o(1))(1-\delta)}^2 \geq (1-\tfrac34 \eps)$, by the union bound this implies,
still in the exponential model $G'$,
\begin{equation}\label{eq:not-too-close}
	\Prob\Big( \exists s' \in N_s, t' \in N_t \colon
           d_{G'}(s',t') \leq (1-\tfrac34 \eps) \ln n / n \Big)
       \; \leq \; {\left(\nto \right)}^2 n^{-\Theta(1)}
         \; = \; o(1) .
\end{equation}
By standard coupling arguments (see \cref{remark:blackbox}), this also implies that \cref{eq:not-too-close} holds in the \emph{uniform} model $G$ in which we are working.
 	
Thus \whp, for all $s' \in N_s, t' \in N_t$, we have
$d_{G'}(s',t') \geq (1-\tfrac34 \eps) \ln n$; assume this holds.
We already assumed that each path $P_i$, $i\leq k$,
goes via some $s' \in N_s, t' \in N_t$,
so its non-root edges contribute at least
$d_{G'}(s',t') \geq (1-\tfrac34 \eps) \ln n/n$ to $S_k$.
Then,
for all $k$ in this range,
\begin{align}
S_k
 &\geq \sum_{i=1}^{k} (1-\tfrac34 \eps) \frac{\ln n}n
    + \sum_{i=1}^{k-1} \left( \Woi^s + \Woi^t \right) \notag
 \\& \geq \sum_{i=1}^{k} (1-\tfrac34 \eps) \frac{\ln n}n
    + (1-\tfrac12 \eps) \sum_{i=\Klower}^{k-1} \frac{2i}n
    \eqnote{by \cref{eq:edge-orderstat}}
    \label{lbx1}
 \\& \geq 	\SkLower.  \notag
\end{align}
To justify the final inequality,
rewrite the second sum in \cref{lbx1} as
$\sum_{i=1}^{k}  \frac{2i}n
    - \frac{2k}n
    - \sum_{i=1}^{\Klower-1} \frac{2i}n$
and observe that both its second term, $2k/n$,
and its final term, which is of order $O(\Klower^2/n)$,
are negligible compared with the first sum in \cref{lbx1},
which is of order $\Omega(k \ln n/n)$.
\end{proof}

\subsection{Lower bound for large \tp{$k$}{k}} \label{LBlargek}

\begin{lemma}\label{lemma:X_k-lower}
	For any $\eps>0$, \whp, simultaneously for every $k \in [\sqrt{\ln n}, n-1]$,
\[ X_k \geq (1-\eps)\parens{\frac{2k + \ln n}{n} }. \]
\end{lemma}

\begin{proof}
Let $\delta=\eps^2/9$ and define
\begin{equation}
	c_k = \frac{2k+\ln n}{n},
    \quad L_k = (1-\delta)\sum_{i=1}^k c_i,
    \quad U_k = (1+\delta) \sum_{i=1}^k c_i.
\end{equation}
\Whp, simultaneously for all $k$,
$S_k \geq L_k$ (by \cref{lemma:S_k-lower})
and $S_k \leq U_k$ (by the upper bound of~\cref{Tmain}, already proved).
Henceforth, assume that both hold, so $L_k \leq S_k \leq U_k$.
The rest of the argument is deterministic.
For any positive integer $t<k$, using that $X_k$ is monotone increasing, we have
\begin{align}
t X_k
  & \geq   X_k +\cdots+ X_{k-t+1}   \notag
  \\& =    S_k - S_{k-t}            \notag
  \\& \geq L_k - U_{k-t}            \label{LBform}
  .
\end{align}
Thus
\begin{align*}
X_k &\geq \frac1t \parens{L_k - U_{k-t}}
 = \frac1t \parens{ (1-\delta)\sum_{i=1}^k c_i - (1+\delta)\sum_{i=1}^{k-t} c_i }
 \\&\geq 	\frac1t \parens{ \sum_{i=k-t+1}^k c_i -2\delta \sum_{i=1}^{k} c_i }
 \geq \frac1t \parens{ t c_{k-t} - 2\delta k c_k }
 = c_{k-t} - \frac{2 \delta k c_k}{t}
 \\&= c_k- \frac{2t}{n} - \frac{2 \delta k c_k}{t}
 \\& \geq c_k- \frac{t c_k}{k} - \frac{2\delta k}{t} c_k
    \quad\text{(using that $c_k/k>2/n$)}
 \\&= c_k \parens{ 1- \frac{t}{k} - \frac{2\delta k}{t} }
 .
\end{align*}
Ignoring integrality for a moment, setting $t=k\sqrt{2\delta}$
would make the last expression $c_k (1- 2\sqrt{2\delta})$.
Since this $t=\Theta(k)=\omega(1)$, rounding it can be seen to
change the expression by a factor $1+o(1)$, so we may safely write
\begin{align*}
X_k &\geq c_k (1-3\sqrt{\delta})
    = (1-\eps) \kcost.
\end{align*}

\end{proof}

\section{Exponential model} \label{ExpBounds}
In this section we prove \cref{Texp}, the analogue of \cref{Tmain} for exponentially distributed edge weights.

For small $k$, results for the exponential case follow from those for the uniform.
We first argue that
the upper bound of \cref{Tmain} also holds in the exponential case
for any $k=o(n)$.
Couple the two models, so that any edge of weight $w=o(1)$ in one model
has cost $w'=w(1+o(1))$ in the other.
The uniform-model upper-bound constructions in
\cref{sec:k-small} (for $k=o(n^{1/2})$)
and \cref{largekUB,unifUB} (for larger $k$)
only use edges of weight $o(1)$ (when $k=o(n)$),
and therefore the same upper bounds hold for the exponential model;
the multiplicative difference of $1+o(1)$ can be subsumed into the
factor $1+\eps$ already present.
(In the construction of \cref{largekUB,unifUB},
the ``middle edges'' are of cost $o(1)$ for \emph{all} $k$,
but the ``incident edges'' have larger cost for $k$ large.
In particular, for large $k$,
\cref{eq:Wkk0} will no longer hold in the exponential case
until we adjust $\rr0$ and $\eps_k$ appropriately.)

For the lower bound too, the argument in \cref{lowerbound} carries
over for all $k=o(n)$.
The lower bounds $L_k$ on the prefix sums $S_k$
derived in \cref{LBsmallk,LBrunning}
carry over to the exponential case
because the edge costs are equal to within $1+o(1)$ factors in the two models.
The upper bounds $U_k$ on the prefix sums are simply the sums of the
individual upper bounds on $X_k$, and we have just argued that these change
only by a $1+o(1)$ factor.
\cref{LBlargek} only uses $L_k$ and $U_k$ to derive lower bounds on $X_k$,
so with these both changed only by $1+o(1)$ factors,
its results carry over verbatim.

Our task, then, is to prove the upper and lower bounds in \cref{Texp}
for larger $k$.
For the upper bound,
arguing for $k> n^{0.4}$ (there is no advantage to a larger starting value),
we use the same approach as for the uniform model in \cref{largekUB}.

For the lower bound, we argue for $k \geq n^{9/10}$.
Unfortunately, the method used in \cref{lowerbound}
for the uniform distribution does not extend;
let us explain why.
The lower bound there came from \cref{LBform},
$t X_k \geq L_k - U_{k-t}$,
valid for any functions $L$ and $U$ with $L_k \leq S_k \leq U_k$.
Here, we would take $L_k$ as the sum $I_k$ of incident edges as in \cref{Ik}
and $U_k$ as the sum of the $X_k$ upper bounds as in \cref{XkWok}.
Recall that we defined $B_k$
so that $B_k \geq U_k-L_k$, as in \cref{budget}.
Then we can rewrite the previous lower bound approach as
$X_k
 \geq \frac1t (L_k-U_{k-t})
 \geq \frac1t (L_k-L_{k-t}) + \frac1t (L_{k-t}-U_{k-t})
 \geq \frac1t \sum_{i=k-t}^{k-1} \Woi - \frac1t B_{k-t}$.
For large $k$,
$\Wok$ and therefore $X_k$ are $\Theta(\ln n)$.
Since the $B_k$ grow to size $\Theta(n^{1/2})$
(in the exponential case as well as the uniform case),
we are thus limited by the second term to
$t=\Omega(n^{1/2+o(1)})$.
However, from \cref{muX}, such a large value of $t$
would mean that the average given by the first term
is significantly different from $\Wok$.

The desired lower bound would be immediate if we could claim that
$P_k$ necessarily used the $k$th cheapest edge on $s$ (of cost $\Wok^s$)
or a later one, and likewise for $t$.
We will prove something close to this.
We argue in \cref{ExpLB} that every pair of vertices (excluding both $s$ and $t$)
is joined by a path of cost at most $\delta$
(for some small $\delta$ to be specified)
that is edge-disjoint from \emph{all} $P_i$, $i=1,\ldots,n-1$.
We will show that this implies that path $P_k$ uses an edge on $s$ that is at most $\delta$ cheaper than $\Wok^s$,
and likewise for $t$,
yielding a sufficient lower bound.

\subsection{Claims, and implications for \cref{Texp}} \label{ExpUpper}
In order to establish upper bounds on $X_k$ in the exponential model,
we use the same structure $\Rk$ as described in \cref{structR}.
Then \cref{XkWok} follows as before,
and we can continue to define $U_k$ as in \cref{Ukdef}.
For convenience define
\begin{align}\label{kbardef}
  \kbar &= n-k .
\end{align}

As before we will treat $k$ in two ranges,
and we start now with the smaller range.

\begin{claim}\label{expkmedium}
	For $k \in [n^{4/10}, \expkstar \,]$, let
	\begin{align}\label{ExpBk}
	B_k \coloneqq \parens{ \frac{2 n^{1/25}+\CB(n^{3/5}- \kbar^{3/5})}{n^{1/5}} }^{5/4}
	\quad \text{and} \quad
	\eps_k \coloneqq \Ce B_k^{1/5} n^{-1/5} \kbar^{-2/5} ,
	\end{align}
	with $\CB=44$ and $\Ce=4$.
	Then, asymptotically almost surely,
	\begin{equation}\label{ExpXkWk}
	X_\kp \leq \Wokp^s + \Wokp^t + 8 \epsk.
	\end{equation}
\end{claim}

\noindent\textbf{Remark:}
In proving \cref{expkmedium} we will set
\begin{align}\label{expk0med}
\rr0 & \coloneqq \epsk \kbar .
\end{align}
because it roughly equates $\Wo{k+\rr0}-\Wok$ and $\epsk$; see \cref{muX}.
In this regime integrality is not an issue:
$\rr0$ is large, per \cref{eq:Expr0}.

It is clear that both $B_k$ and $\eps_k$ in \cref{ExpBk} are increasing in $k$,
even over the larger range $k \in [0,n]$.
We will make use of the following bounds,
holding for $n$ sufficiently large.
Here, \cref{eq:ExpBkUpper} uses that at $k=n-\Theta(\sqrt{n})$,
$\kbar^{3/5}$ dominates $2n^{1/25}$,
while \cref{eq:ExpBkLower} takes $k=0$.
\begin{align}
	B_k &\leq B_{\expkstar} \leq \CB^{5/4} n^{1/2} \label{eq:ExpBkUpper}\\
	B_k &\geq B_{n^{4/10}} \geq 2n^{-{1/5}} \label{eq:ExpBkLower}
\\
	\epsk & \leq \Ce B_k^{1/5} n^{-1/5} n^{-1/2 \cdot 2/5} \leq \Ce \CB^{1/4} n^{-3/10} \label{eq:ExpEkUpper} \\
	\epsk & \geq \Ce {(B_{n^{4/10}})}^{1/5} \, n^{-1/5} \, \kbar^{-2/5} \geq \Ce n^{-{6/25}} \kbar^{-2/5}  \label{eq:ExpEkLower}
\\
\rr0 &= \kbar \epsk
  \stackrel{\cref{eq:ExpEkLower}}{\geq} \Ce n^{-6/25} \kbar^{3/5}
  \geq \Ce n^{3/50} . \label{eq:Expr0}
\end{align}

\begin{claim}\label{expkbig}
For $k \in  (\expkstar, n-2 \,] $, let
\begin{align}\label{ExpBk2}
B_k \coloneqq \CB' \sqrt n
\quad \text{and} \quad
\eps_k \coloneqq C'_\eps n^{-1/6} ,
\end{align}
with $\CB=115$ and $\Ce=5$.
Then, asymptotically almost surely,
simultaneously for all $k$ in this range,
\begin{align}
	X_\kp &\leq \Wokp^s + \Wokp^t + 8 \epsk. \label{ExpXkUB}
\end{align}
\end{claim}

\noindent\textbf{Remark:}
In proving \cref{expkbig} we will set
\begin{align}\label{expk0big}
\rr0 & \coloneqq 1 .
\end{align}

As in \cref{kbig}, $B_k$ and $\eps_k$ are constants independent of $k$,
but we retain the subscript for consistency with the notation of \cref{largekIntro}.

\begin{proof}[Proof of the upper bounds in \cref{Texp}]
Analogous to the argument in \cref{uniclaims}, it is sufficient to check that $\epsk = o(\E \Wok)$.
Since $\E \Wok \sim \ln \parens{\frac{n}{n-k}} \geq \frac kn$ (see \cref{muX}),
it is enough to show that $\eps_k = o(k/n).$

For $k \leq n^{0.99} = o(n)$, by first-order approximation,
\begin{align}\label{n35}
  n^{3/5}-\kbar^{3/5}
   &\eqdef n^{3/5}-(n-k)^{3/5}
   \sim \tfrac 35 n^{-2/5} k ,
\end{align}
so
$ B_k = \Theta \parens{n^{-1/5} + n^{-3/4} k^{5/4}} $.
Hence, from \cref{expkmedium}, specifically \cref{ExpBk},
\begin{align}
\eps_k
= \Theta( (n^{-1/25} + n^{-3/20}k^{1/4}) n^{-1/5} n^{-2/5} )
= \Theta( n^{-16/25} + n^{-3/4}k^{1/4})
= o(k/n)
\end{align}
as $k \geq n^{4/10}$.

For $k > n^{0.99}$, we have in \cref{expkmedium} that $\eps_k = O(n^{-3/10})$ by \cref{eq:ExpBkUpper}, and so $\eps_k = o(k/n)$,
while in \cref{expkbig},
$\eps_k = \Theta(n^{-1/6}) =o(k/n)$.
\end{proof}

\subsection{Path weights}

To show inequality \cref{path+} it suffices to show that
\begin{align}\label{eq:expWkk0}
\Wokk - \Wokp \leq 1.1 \epsk.
\end{align}
In \cref{expkbig}, we have defined $\rr0 \coloneqq 1$, so \cref{eq:expWkk0} is trivial.
For \cref{expkmedium},
$\Delta \coloneqq \Wokk-\Wokp$
has the same distribution as $\sum_{i=k+2}^{k+\rr0} X(n-i)$,
where $X(a) \sim \Exp(a)$
and these variables are all independent.
Thus $\Delta$ is stochastically dominated by the sum of $\rr0-1$
independent random variables $X(\kbar-\rr0)$.
Since $\rr0 = \kbar \eps_k$,
we have that
$\E \Delta \leq \rr0/(\kbar-\rr0) = \epsk/(1-\epsk)$,
and from \cref{exptail} it follows that
$\Pr(\Delta > 1.1 \epsk) = O(\exp(-\Theta(\rr0)))$.
From \cref{eq:Expr0},
by the union bound, there is a negligible chance that \cref{path+}
fails in any round.

\subsection{Budgets in \cref{expkmedium}} \label{ExpC1budgets}
As before, we need to define a $B_k$ satisfying \cref{budget}
and, as before, $\epsk$ can be guessed from \cref{epsk},
then checked to satisfy yield robustness as in \cref{kmedrobust,kbigrobust}.
The base case, confirming \cref{Bkbase}, is given by
$k=n^{4/10}$,
where by \cref{eq:ExpBkLower}
\begin{align}
B_k
\geq 2 n^{-1/5}
\eqdef U_k .
\end{align}

To verify \cref{Bksuff},
it is straightforward to check that
$\frac{\partial^2}{\partial k^2} \Bk$ is positive,
so
$\frac{\partial}{\partial k} \Bk$ is increasing,
and
\[
B_\kp-\Bk \geq \frac{\partial}{\partial k} \Bk
  = \frac 54 \frac 35 \frac {\CB}{\Ce}
  = \frac 34 \frac {\CB}{\Ce}
 \geq 8\epsk ,
\]
since $\frac 34 \frac {\CB}{\Ce} \geq 8$ (check).
Finally, we establish \cref{epsfit}.
We show that \whp for all $k$ in the range,
\begin{align}\label{expepsfitpf}
 \Delta
    &\coloneqq \Wo \kp^s-\Wok^s
    \leq 0.1 \epsk .
\end{align}
Note that
$\Delta \sim \Exp(\kbar-1)$, so
\begin{align*}
\Pr \parens{\Delta > 0.1\eps_k} =
\exp \parens{-0.1 \epsk \cdot (\kbar-1)}
= \exp(-\Omega(\rr0))
= \exp(-n^{\Omega(1)})
\end{align*}
by \cref{eq:Expr0}.
Then, by the union bound there is a negligible chance that
\cref{expepsfitpf} fails for any $k$.

\subsection{Robustness in \cref{expkmedium}}
With reference to \cref{robustness}, we complete the robustness argument for \cref{expkmedium}, showing that \cref{robustblurb} holds with high probability.
Here we have taken $\rr0 = \epsk \kbar$,
so the number of edges from a middle vertex to $V'_S$
(see \cref{Zvs})
is $Z^s_v  \sim \Bi(\epsk \kbar, \epsk)$,
with mean
\begin{align}\label{Expla}
  \la &= \rr0 \epsk = \epsk^2 \kbar
\end{align}
(see \cref{kmedlambda}).
Recall that if $\la$ is small we expect (see \cref{kmedtotalweight})
that to destroy all paths the adversary will have to delete
edges of total weight at least
$\epsk \, n \, \lambda^2 = \epsk^5 n \kbar^2$,
which will exceed $B_k$.
And, if $\la$ is large, then each $Z_v$ will have expectation close to $\la=\epsk^2 \kbar$, for a total cost $\epsk n$ times larger, namely $\epsk^3 n \kbar$,
and again this exceeds $B_k$.

We now show the details of these rough calculations,
including the probabilistic details,
applying \cref{lemma:BinMin} to $Z_v$ in the
two cases of $\la$ small and large.

For the adversary to delete all \st paths via $v$, he must delete at least
\[ Z_v\coloneqq \min(Z_v^s, Z_v^t) \]
edges, and to destroy all paths he must delete at least
\[ N\coloneqq \sum_{v \in M'} Z_v \]
edges.
As described in \cref{robustness},
we imagine a fixed deletion of $k$ edges on each of $s$ and $t$,
giving neighbour sets $V'_s$ and $V'_t$
and a set $M'$ of middle vertices,
eventually taking a union bound over all such choices.

\medskip \noindent \tmbf{If $\lambda \geq 2$}, then
by \cref{lemma:BinMin},
for each $v \in M'$, $\Pr(Z^s_v \geq 0.65 \la) \geq 1/4$.
Thus, $N$ stochastically dominates $0.65\lambda \cdot \Bi(0.99n, 1/4)$,
with expectation $> 0.1608 \la n$.
We shall consider it a \emph{failure} if $N \leq 0.16 \lambda n$.
Assuming success,
since each edge costs at least $\epsk$ to delete,
it costs at least $0.16 \epsk \la n = 0.16 \epsk^3 n \kbar$ to delete them all.
This exceeds $B_k$:
\begin{align*}
\frac{0.16 \, \epsk^3 n \kbar}{B_k}
  &= 0.16 \, \Ce^3 n^{-3/5} \kbar^{-6/5} B_k^{-2/5} n \kbar
   \eqnote{by definition of $\eps_k$}
\\&= 0.16 \, \Ce^3 B_k^{-2/5} n^{2/5} \kbar^{-1/5}
\\& \geq 0.15 \, \Ce^3 \CB^{-1/2} n^{1/5} n^{-1/5}
    \eqnote{by~\cref{eq:ExpBkUpper}}
\\ &> 1,
\end{align*}
using that $0.15 \cdot \Ce^3 \CB^{-1/2}>1$ (check).

Failure means that $N/(0.65 \la) \sim \Bi(0.99n, 1/4) \leq (0.16\la n)/(0.65 n)=(0.16/0.65)n$.
Noting that $0.99 \cdot 1/4 > 0.16/0.65$,
by \cref{lemma:BinDev}, the probability of failure is
$\exp(-\Omega(n))$.
By the union bound,
the total of the failure probabilities,
over all rounds and all adversary choices of the $k$ root edges at $s$ and $t$, is small:
\begin{align} \label{expcase1failure}
\sum_k {{\binom{\kk}{\rr0}}^2} & \cdot \exp(-\Omega(n))
\\& \leq \sum_k \, {(n^{\rr0})}^2 \exp(-\Omega(n))
\notag
\\ &= \sum_k \exp\parens{2\epsk \kbar \ln n-\Omega(n)}
\quad\text{(by $\rr0=\epsk \kbar$)}  \notag
\\&\leq n \exp(-\Omega(n)) = o(1), \notag
\end{align}
the penultimate inequality using $\epsk \kbar = O(n^{7/10})$ by~\cref{eq:ExpEkUpper}.

\medskip \noindent \tmbf{If $\lambda < 2$},
then by \cref{lemma:BinMin} $N$ stochastically dominates $\Bi(0.99n, 0.18\lambda^2)$,
with expectation $> 0.175 \la^2 n$.
We shall consider it a \emph{failure}
if $N \leq 0.17 \lambda^2 n = 0.17 \epsk^4 n \kbar^2$.
Each edge costs at least $\epsk$ to delete.
Assuming success, it thus costs at least $0.17 \epsk^5 n \kbar^2$ to delete them all, which exceeds $B_k$:
\begin{align*}
\frac{0.17 \epsk^5 n \kbar^2}{B_k}
&= 0.17 \Ce^5  \eqnote{by definition of $\eps_k$}
\\ &> 1,
\end{align*}
using that $0.17 \Ce^5>1$ (check).

By \cref{lemma:BinDev}, the probability of {failure} is
\begin{equation}\label{eq:expN2}
\Prob \parens{ N \leq 0.17 \epsk^4 n \kbar^2} = \exp(-\Omega(\epsk^4 n \kbar^2)).
\end{equation}
Over all rounds and adversary choices of edges incident to $s$ and $t$,
the total failure probability is at most
\begin{align}
\sum_k {\binom{\kk}{\rr0}}^2 & \cdot \Prob \parens{N < 0.17 \epsk^4 n \kbar^2}  \notag
\\ &\leq \sum_k \exp\parens{2\epsk \kbar \ln n- \exp(-\Omega(\epsk^4 n \kbar^2))}  \notag
\\&\leq n \exp(-\Omega(\epsk^4 n \kbar^2)) , \notag
\intertext{%
because
$\epsk^4 n \kbar^2$ is larger than $\epsk \kbar$
by a factor $\epsk^3 n \kbar$,
	which by~\cref{eq:ExpEkLower} is
	$\Omega(n^{-18/25} \kbar^{-6/5} n \kbar)=\Omega(n^{7/25} \kbar^{-1/5}) = \Omega(n^{2/25})$. Continuing, this is}
\notag \\
&\leq n \exp(-\Omega(n^{1/25} \, \kbar^{2/5}))
    \eqnote{invoking~\cref{eq:ExpEkLower} again}  \label{expcase2failure}
\\& = o(1) .  \notag
\end{align}

\subsection{Budgets in \cref{expkbig}} \label{expbudgetsbig}
We now establish \cref{budget} for the parameters of \cref{expkbig}.
\cref{ExpC1budgets} showed that \cref{budget} holds
for $k$ up to $\kstar \coloneqq \floor{\expkstar}$,
the point where \cref{expkmedium} ends and just before \cref{expkbig} begins,
so in particular $B_\kstar \geq U_\kstar - I_\kstar$.
For the regime of \cref{expkbig}, we redefine $I_k$ from \cref{Ik}.
Recall that $I_k$ is a lower bound on the edges incident to $s$ and $t$ used by the first $k$ paths. Previously, the sum defining $I_k$ in \cref{Ik} went to $k-1$ to avoid double counting the \ste edge.
In this regime, however, we need the sum to go $k$, as the $\Woi$ increase rapidly.
The weight of the \ste edge is distributed as $\Exp(1)$, thus \whp it costs at most
$\stebound$.
For $k>\kstar$, define
\begin{align} \label{Ikdefnew}
I_k & \coloneqq \sum_{i=1}^k \parens{\Wok^s + \Wok^t} - \stebound,
\end{align}
so that \whp $I_k$ is a lower bound on the incident edges:
the $\stebound$ term resolves the
potential double-counting of \ste.
We are now ready to check that \cref{budget} holds.
Following the derivation of \cref{budgetsClaim2},
for $k$ from $\kstar+1$ to $n-2$,
\begin{align}
U_k - I_k
&= (U_\kstar - I_\kstar) + [(U_k-U_\kstar) - (I_k - I_\kstar)]
\notag \\ &\leq B_\kstar + \sum_{k=\kstar+1}^{n-2} {7 \epsp}
   - (\Wo{\kstar}^s+\Wo{\kstar}^t-\stebound)
  \eqnote{see \cref{Ukdef}, \cref{Ik}, and \cref{Ikdefnew}}
\notag \\&\leq 114 \sqrt n + \sqrt n \cdot 7 C'_\eps n^{-1/6} + \stebound
    \eqnote{see \cref{eq:ExpBkUpper} and \cref{ExpBk2}}
\notag \\& \leq 115 \sqrt n
\notag \\& \leq \Bp  \eqnote{see \cref{ExpBk2}}, \label{ExpBudget}
\end{align}
using that $\CB' \geq 115$ (check).

\subsection{Robustness in \cref{expkbig}} \label{exprobustbig}
Again, our aim is to establish robustness of $R$ by
showing that \cref{robustblurb} holds with high probability,
and the argument is similar to but simpler than that for robustness in \cref{expkmedium}.

Since $\rr0=1$, both $V'_s$ and $V'_t$ have size 1.
For a vertex $v \in M'$, let $Z_v$ be the number of paths from $V'_s$ to $V'_t$ via $v$. There is only one such possible path, hence
\[ Z_v \sim \Bern \parens{\epsp^2}. \]
To destroy all \st paths the adversary must delete at least
\[ N\coloneqq \sum_{v \in M'} Z_v \]
edges. $N$ stochastically dominates $\Bi(0.99n, \epsp^2)$,
with expectation at least $0.99 \epsp^2$.
We declare the event $N \leq 0.98 \epsp^2 n$ a \emph{failure}.
Assuming success, destroying all \st paths would cost at least
$\epsp N \geq 0.98 \epsp^3 n$.
This exceeds $\Bp$, since by \cref{ExpBk2}
and $0.98 {C'_\eps}^3 > C'_B$ (check),
\begin{equation*}
\frac{0.98 \epsp^3 n}{B_k} = \frac{0.98 {C'_\eps}^3}{C'_B}
 >1 .
\end{equation*}

The probability of failure is
\begin{equation}\label{eq:expN3}
\Prob \parens{ N \leq 0.98 \epsp^2 n}
= \exp(-\Omega(\epsp^2 n))
= \exp(-\Omega(n^{2/3})).
\end{equation}

Over all rounds and adversary choices,
using that $\binom{\kk}{\rr0} = \binom{k+1}1 \leq n$,
the total failure probability is at most
\begin{align}
\sum_k {\binom{\kk}{\rr0}}^2 & \cdot \Prob(N \leq 0.98 \eps^2 n) \notag
\\ &\leq
\sqrt n \, n^2 \, \exp (-\Omega(n^{2/3}))  \eqnote{by~\cref{eq:expN3}} \label{eq:expfailure3}
\\& = o(1) . \notag
\end{align}

\subsection{Lower bound} \label{ExpLB}

As argued in the introduction of this section,
for any $k=o(n)$,
the lower bound follows from the uniform case. Thus it is sufficient if we show the lower bound for $k \geq n^{9/10}$, which we do now.

\begin{remark}\label{remdelta}
With high probability, for every pair of vertices $u$ and $v$ in $G' = G-s-t$,
there is a $u$--$v$ path in $G'$ of cost at most
$\delta = 20n^{-1/6}$
that is edge-disjoint from $P_1, \ldots, P_{n-1}$.
\end{remark}

\begin{proof}
The proof of \cref{expkbig}
showed that \whp, for all $k$ in the claim's range (up to $k=n-2$),
there is a cheap \st path (of cost given by \cref{ExpXkUB})
disjoint from $P_1,\ldots,P_k$,
\emph{because} for a given pair of neighbours $u,v$ of $s$ and $t$,
there is a $u$--$v$ path in $G'$ that is edge-disjoint from
these $k$ paths
and has cost at most $4 \epsk = 20n^{-1/6} \eqdef \delta$ (see \cref{ExpBk2}).
The existence of a $k+1$st \st path limits $k$ to $n-2$
since after that there are no new neighbours $u$ and $v$ of $s$ and $t$,
but the rest of the argument extends to $k=\nm$.

In particular, extending the definition \cref{ExpBk2} of $B_k$ and $\epsk$
to $k=\nm$,
the derivation of \cref{ExpBudget} extends without change and shows that
the budget $B_\nm$ covers the middle edges of all paths $P_1,\ldots,P_\nm$,
and the robustness argument also extends and shows \cref{eq:expN3} to hold for $k=\nm$.
Since the failure probability in \cref{eq:expN3} is exponentially small,
and there are fewer than $n^2$ pairs $\set{u,v}$ in $G'$,
\whp there is a cheap path (of cost $\leq \delta$) for every pair.
\end{proof}

For the remainder of this section we assume that the high-probability conclusion
of \cref{remdelta} holds.

Let $H_k^s$ be the weight of the heaviest edge incident to $s$ used by the first $k$ paths,
and let $L_k^s$ be the weight of the lightest edge incident to $s$ \emph{not} used by the first $k$ paths.
Define $H_k^t$ and $L_k^t$ likewise.

We claim that for all $k$ from 1 to $\nm$,
with $\delta=20n^{-1/6}$ as in \cref{remdelta},
\begin{align}\label{HLdelta}
	H_k^s - L_k^s \leq \delta .
\end{align}
We argue by contradiction.
Given $k$, let $P_i$, $i\leq k$, be the path using the edge of weight $H^s_k$.
By \cref{remdelta}, we can construct an \st path $Q$
whose $s$-incident edge is the one of weight $L^s_k$,
whose $t$-incident edge is the same as that of $P_i$,
and whose middle edges cost at most $\delta$
and are not used in $P_1,\ldots,P_{n-1}$.
This path $Q$ is cheaper than $P_i$:
its $s$-incident edge is cheaper by $H^s_k-L^s_k > \delta$,
its $t$-incident edge has the same cost,
and its middle edges
(costing at most $\delta$)
cost at most $\delta$ more than those of $P_i$.
Also, $Q$ is edge-disjoint from the first $i-1$ paths:
its $s$-incident edge
$L^s_k$ is not used even by the first $k$ paths,
the middle edges are disjoint from those of all $n-1$ paths,
and its $t$-incident edge is that used by $P_i$ (so not used by a previous path).
Thus, $Q$ should have been chosen in preference to $P_i$,
a contradiction, establishing \cref{HLdelta}.

Trivially,
$H_k^s \geq \Wok^s$.
Thus, from \cref{HLdelta},
\begin{align}\label{LW}
L^s_k
\geq H_k^s - \delta
\geq \Wok^s - \delta .
\end{align}

For $k \leq n-2$, the edge of $P_\kp$ incident to $s$ costs at least $L_k^s$
and the edge incident to $t$ at least $L_k^t$.
If $P_\kp$ is not the single-edge path $\set{s,t}$ these two edges are distinct,
so that $X_\kp \geq L^s_k + L^t_k$.
If $P_\kp$ is the single-edge path $\set{s,t}$ then $P_k$ is not,
and $X_\kp \geq X_k \geq L^s_\km + L^t_\km$.
Either way, by \cref{LW},
\begin{align}
X_{k+1}	&\geq L^s_\km + L^t_\km \notag
\\		&\geq \Wo\km^s + \Wo\km^t - 2\delta.  \label{bigkLB1}
\end{align}

Recall that we are concerned here with $k \geq n^{9/10}$.
By \cref{lemma:edge-orderstat}, for all such $k$, and for any $\gamma>0$,
\whp $\Wok \geq (1-\gamma) \EWW k$.
Since the exponential random variable is stochastically greater than the uniform,
$\EWW k > k/n = \Omega(n^{-1/10})$,
while
$\delta = 20n^{-1/6} = o(\EWW k)$.
From \cref{muX} it is clear that $\EWW \km \asymp \EWW \kp$
(for any $k=\omega(1)$),
and we subsume the asymptotic error into the constant $\gamma$.
Thus, from \cref{bigkLB1}, for any $\gamma>0$, \whp, for all $k \geq n^{9/10}$,
\begin{align*}
 X_k \geq (1-\gamma) 2 \EWW k ,
\end{align*}
completing the proof of the lower bound in \cref{Texp}.

\section{Expectation}\label{sec:expectation}

In this section we prove \cref{thm:expectation}.
We treat the uniform and exponential models at the same time.
Let $\evp k$ be the event that $P_k$ exists.
Clearly $\Pr(\evp k) \geq \Pr(\evp {n-1})$.
By \cref{Tmain} (for the uniformly random model)
and \cref{Texp} (for the exponential model),
$\Pr(\evp {n-1}) = 1-o(1)$.
This establishes the first part of the theorem.
Then, let $\mu_k = 2\E \Wok + \ln n /n$
(so for the uniform model, $\mu_k=\wo(k)$).
It suffices to show that
\begin{align}
\E[X_k \mid \evp k] &= (1+o(1)) \mu_k  \label{Emu}
\end{align}
uniformly in $k$.

First, we show the lower bound
implicit in \cref{Emu}.
Fix $\eps > 0$.
Let $\evl k$ be the event that (jointly)
$P_k$ exists and $X_k \geq (1-\eps) \mu_k$.
By \cref{Tmain} (for the uniform model)
and \cref{Texp} (for the exponential model),
$\evl k$ holds with probability $1-o(1)$ uniformly in $k$.
Thus,
\begin{align*}
E[X_k \mid \evp k]
 &\geq \Pr(\evl k) \E[X_k \mid \evp k \wedge \evl k ]
 \geq (1-o(1)) \, (1-\eps) \mu_k.
\end{align*}
Since this holds for any $\eps$, we have that
\[ \E[X_k \mid \evp k] \geq (1-o(1)) \mu_k. \]

We now establish the corresponding upper bound.

\subsection{Small \tp{$k$}{k}} \label{expSmallk}
First, we consider the range $k \leq n^{4/10}$.
We will need the following lemma in \cref{EXkU}.

\begin{lemma}\label{lem:large-eps}
There exists an absolute constant $C>0$ such that, for all $\eps>C$,
in both the exponential and uniform models,
for all $k=o(\sqrt n)$ the probability of the event
\begin{align}\label{eq:indC}
	X_k > (1+\eps) \mu_k
\end{align}
is $\OO{n^{-1.9}}$.
\end{lemma}
\begin{proof}
By the reasoning given in the introduction of \cref{ExpBounds}, it is sufficient to show the result in the uniform case, where $\mu_k=\kcost$.
We use the same argument as developed in \cref{sec:k-small}, where we prove \cref{Tmain} up to $k=o(\sqrt n)$.
Our argument in \cref{sec:k-small}
(see \cref{indHyp})
was that
for any sufficiently small $\eps>0$,
\begin{align}\label{eq:ind-large}
\text{if }
	X_i \leq \ope \parens { \frac{2i}n+\frac{\ln n}n }
\text{for all $i \leq k$, then \whp the same holds for $i=k+1$.}
\end{align}
We proved this by constructing a structure $R=\Rk$ in $G$,
in which after deleting $k$ paths,
each of cost $\leq \ope (2k/n+\ln n/n)$ from $G$,
\whp there remains a path in $R$ satisfying the same cost bound.
By \cref{ksucceeds}, the probability of failure was
$\OO{n^{-1.9}}+\exp(-\Theta(s(k)))$.
This does not suffice since for $k$ small the second term
may exceed $\OO{n^{-1.9}}$ (recall $s = 2k + \ln n$).
	
To prove the lemma, we will show that
for some sufficiently \emph{large} constant $\eps$,
the failure probability in \cref{eq:ind-large} is $\OO{n^{-1.9}}$.
As noted in \cref{warning}, a few parts of the
argument developed in \cref{sec:k-small}
rely on $\eps$ being sufficiently small,
and here we will detail the changes needed.
Principally, we will make one modification (a simplification)
to \cref{sec:k-small}'s construction of $R$.
We will also track the dependence
of key Landau-notation expressions on $\eps$.

\medskip

Recall from \cref{sdef,w0def} that $s=2k + \ln n$ and $w_0 = s/n$.

Parallelling the structure of \cref{sec:k-small},
we start by reviewing the adversary's edge-count budget.
This was given by \cref{Bany} which,
through its dependence on \cref{pathlength},
held only for sufficiently {small} $\eps$.
For sufficiently large $\eps$,
modulo the one-time failure probability $\OO{n^{-1.9}}$ from \cref{LLenBd},
each of the first $k$ paths has length
$\leq (1+\eps) w_0 \cdot 19 n < 20 s \eps$,
and the total length of the first $k$ paths is at most
\begin{align}\label{eq:budget-large}
	20 k s \eps < 10 s^2 \eps ,
\end{align}
so we now take this to be the adversary's budget.

We build level-0 edges of $R$ exactly as in \cref{level0},
and using the same parameter $\rr0$.
That is, we add the cheapest $\kk$ edges incident on $s$,
with $\rr0= \ceil{\tfrac1{10} \eps s}$ as in \cref{k0};
the opposite endpoints of these edges are the level-1 vertices.
Recall that we declared this step a failure if the number $X$ of edges with weights in the interval $[0, \tfrac k n+\frac19 \eps \wo]$
is smaller than $\kk$.
Note that $X \sim \Bi(n', \tfrac k n+\frac19 \eps \wo)$, thus $\E X = (1-o(1)) \, (k+\frac19 \eps s)$,
and failure means that $X <\kk$, i.e., that
\begin{align*}
	\frac{X}{\E X}
	= (1+o(1)) \, \frac{k+\frac1{10} \eps s}{k+\frac19 \eps s}
	\leq \frac{10}{11}
\end{align*}
for $\eps$ sufficiently large.
Then, analogously to \cref{level0fail}, the failure probability by \cref{lemma:BinDev} is at most
	\begin{align}
	\Pr(X < \tfrac{10}{11} \E X)
	&\leq \exp(-\Omega(\E X))
	\leq \exp(-\Omega(\eps s)). \label{eq:large-level0fail}
	\end{align}
	
We skip constructing level-1 edges as in \cref{level1},
instead setting the level-2 vertices identical to level-1 vertices.
(There are no edges between these levels;
we have ``level 2'' only to keep the level numbering the same as before.)

We build level-2 edges exactly as before, with the same parameter $\rr2$,
linking to each level-2 vertex its
cheapest $\rr2= \frac{1}{10} \eps s$ neighbours
(which become the level-3 vertices).
The calculations in \cref{level2} hold for any $\eps>0$, and from \cref{level2fail} the probability of any failure on this level is
\begin{align}\label{eq:large-level1fail}
	\leq \exp{-\Theta(\eps s)}.
\end{align}

The adversary's deletions of edges incident on $s$
must leave $r_0$ vertices at level 1 (a.k.a.\ level 2),
thus $\rr0 \rr2 = \eps^2 s^2/100$ edges leading to level~3.
By \cref{eq:budget-large} the adversary
is allowed to delete at most $10 s^2 \eps$ edges,
so for $\eps$ sufficiently large,
at least $2 s^2$ level-3 vertices remain;
this is the same as before,
and will continue to suffice.
	
From level~3 we construct shortest-path trees just as in \cref{level3},
whose calculations hold for any $\eps > 0$.
To recapitulate, these trees are built to a size \cref{ddef}
independent of $\eps$,
the calculations made are valid for all $\eps$,
and the result (here as in \cref{sec:k-small})
is that each tree fails with some probability $o(1)$,
but the level as a whole fails only if at least $0.01 s^2$ trees fail,
which occurs with probability only $\exp(-\Omega(s^2))$
(see \cref{level3fail}).
	
This concludes the modified construction of $R$.
The remainder of the argument is unchanged from \cref{sec:k-small}.
In the absence of failures,
the maximum weight of any \st path in $R$ remains
at most $(1+\eps) w_0$ per \cref{Rpathcost}
(indeed, a little less as we've skipped the level-1 edges).
The number of successful level-3 trees is $\Omega(s^2)$ as before,
and the calculations leading to the
probability that an adversary can destroy all cheap paths in $R$
are unaffected: this probability remains $\exp(-\Omega(s^2 \ln n))$
as in \cref{smallkAdversaryFailure},
which is dominated by other failure probabilities.

Tallying up, as in \cref{sec:small-success}, we have a one-time failure probability of $\OO{n^{-1.9}}$ from \cref{LLenBd}.
Out of levels 0, 2 and 3 we
have failure probabilities given respectively by
\cref{eq:large-level0fail}, \cref{eq:large-level1fail} and \cref{level3fail},
namely
$\exp(-\Omega( \eps s))$,
$\exp(-\Omega( \eps s))$ and
$\exp(-\Omega( s^2))$.
Since $s > \ln n$, for some $\eps$ sufficiently large, the net failure probability is $\OO{n^{-1.9}}$, as claimed.
\end{proof}

Let $C$ be the constant in \cref{lem:large-eps}.
Separately, fix any sufficiently small $\eps>0$.
Let
\begin{align*}
U_1 &= [0, (1+\eps){\mu_k}), \\
U_2 &= [(1+\eps){\mu_k}, C{\mu_k}), \\
U_3 &= [C{\mu_k}, \infty).
\end{align*}
Let $\eva_i$ be the event that $X_k \in U_i$.
By \cref{Tmain}, $\Pr(\eva_1) = 1-o(1)$ and $\Pr(\eva_2)=o(1)$,
and by \cref{lem:large-eps}, $\Pr(\eva_3)=O(n^{-1.9})$.

Since here we are considering $k \leq n^{4/10} \leq \expcutoff$,
with reference to the proof of \cref{rmk:existence},
one possible choice for $P_k$ is some path of length 2
(there must remain at least one such),
and thus, deterministically,
\begin{align}\label{eq:length2}
X_k \leq W_s + W_t ,
\end{align}
where $W_v$ denotes most expensive edge out of $v$
($W_v = W^v_{\os{n-1}}$ in the notation of \cref{Wkv}).

In the uniform model, \cref{eq:length2} means that,
deterministically, $X_k \leq 2$.
Then,
\begin{align}
\E[X_k]
&=\Pr(\eva_1) \E[X_k \mid \eva_1] + \Pr(\eva_2) \E[X_k \mid \eva_2] + \Pr(\eva_3) \E[X_k \mid \eva_3] \notag \\
&\leq (1-o(1)) \cdot (1+\eps) {\mu_k} + o(1) \cdot (1+C){\mu_k} + O(n^{-1.9}) \cdot 2 \notag \\
&\leq (1+\eps+o(1)){\mu_k}  , \label{EXkU}
\end{align}
since $\mu_k > \ln n/n$.
As this holds for arbitrarily small $\eps > 0$,
\begin{align}
\E[X_k] \leq (1+o(1)) {\mu_k} .  \label{eq:expectLB}
\end{align}

For the exponential model the same argument applies,
once we control
$\E[X_k \mid \eva_3]$.
We make use of the following inequality.
Let $Z$ be a random variable with CDF $F$, and  $\eva$ be an event with $\Pr(\eva) = \alpha$.
Then,
\begin{align}\label{eq:F}
\E[Z \mid \eva] \leq \E[Z \mid Z > F^{-1} (1-\alpha) ].
\end{align}
In the case that $Z$ is an exponential random variable with rate $\lambda$,
$F(z)=1-\exp(-\lambda z)$,
so $F^{-1} (1-\alpha) = -\ln(\alpha)/\lambda$.
By the memoryless property of the exponential, the RHS of \cref{eq:F} is
$\E[Z]+F^{-1} (1-\alpha)$,
giving
\begin{align}
 \E[Z \mid \eva]
  & \leq \frac{1-\ln(\alpha)}{\lambda} . \label{ExpCondExp}
\end{align}

Recall from \cref{ordersumexp} that $W_v = \sum_{i=1}^{n-1} Z_i$ where
$Z_i \sim \Exp(i)$.
Condition on the event $\eva_3$,
taking $\alpha \coloneqq \Pr(\eva_3) = O(n^{-1.9})$.
By \cref{ExpCondExp},
\begin{align}
\E[W_k \mid \eva_3]
=\sum_{i=1}^{n-1} \E [Z_i \mid \eva_3]
\leq \sum_{i=1}^{n-1} \frac{1-\ln(\alpha)}{i}
\sim (1-\ln(\alpha)) \ln n
= O(\ln^2 n).  \label{eq:Walpha}
\end{align}
By \cref{eq:length2}, \cref{eq:Walpha} and linearity of expectation,
\begin{align}
\Pr(\eva_3) \E[X_k \mid \eva_3]
\leq \alpha \E[W_s + W_t \mid \eva_3]
= 2 \alpha O(\ln^2 n)
= O(n^{-1.9} \, \ln^2 n) ,
\label{Exp2Path}
\end{align}
which is $o(\mu_k)$
since $\mu_k > \ln n/n$.
Thus \cref{EXkU} holds also for the exponential model
(the change to the middle line of the calculation affects nothing),
whereupon so does \cref{eq:expectLB}.

\subsection{Large \tp{$k$}{k}}
For $k \geq n^{4/10}$, we gather the failure events in \cref{largekUB}.
First, we have $X_{n^{4/10}} \leq 3 n^{4/10} / n$
with failure probability $O(n^{-1.9})$,
from \cref{Xn0.4} and \cref{Pn0.4}.
Then, we have to check two types of failures:
failure of \cref{path+} to be an upper bound on \cref{path1}
(because the edge order statistics are not as expected),
and violation of \cref{XkWok}
(because $R$ fails to be robust against the adversary).

Failure of \cref{path+} as an upper bound is, in the uniform model, checked
through violation of \cref{eq:Wkk0},
the paragraph after \cref{eq:Wkk0} showing failure
to occur w.p.\ at most $\exp(-\Omega(n^{0.01}))$.
Likewise, in the exponential model
it is checked in and following \cref{eq:expWkk0},
with a failure probability of $O(\exp(-\Omega(n^{3/50})))$.

The failure probability of \cref{XkWok}
in the uniform model is calculated for three cases:
near \cref{case1failure} as $n \exp(-\Omega(n))$,
near \cref{case2failure} as $n \exp(-\Omega(n^{11/25}))$,
and near \cref{eq:unifailure3} as $14 n^{5/2} \exp(-\Omega(n^{2/3}))$.
The failure probability
in the exponential model is also calculated for three cases:
near \cref{expcase1failure} as $n \exp(-\Omega(n))$,
near \cref{expcase2failure} as $n \exp(-\Omega(n^{1/25}))$,
and near \cref{eq:expfailure3} $n^{5/2} \exp(-\Omega(n^{2/3}))$.

Thus, the failure probabilities for \cref{path+} and \cref{XkWok}
are all $O(\exp(-n^{0.01}))$,
so the probability of any failure affecting any $k>n^{4/10}$ is $O(n^{-1.9})$.

Let
\begin{align*}
U_1 &= [0, (1+\eps){\mu_k}) \\
U_2 &= [(1+\eps){\mu_k}, \infty),
\end{align*}
and let $\eva_i$ be the event that $P_k$ exists and $X_k \in U_i$.
Thus $\Pr(\eva_1)=1-o(1)$ and $\Pr(\eva_2) = O(n^{-1.9})$.

Conditioning on the event $\evp k$ that $P_k$ exists,
this path clearly has cost
\[ X_k \leq Z \coloneqq \sum_{v \in V(G)} W_v  \]
(analogous to \cref{eq:length2}).
In the uniform model, deterministically, $Z \leq n$.
In the exponential model,
the event $\eva_2$ here has the same probability as event $\eva_3$
in \cref{expSmallk}, so we may reuse \cref{eq:Walpha}, obtaining
\begin{align*}
\E[Z \mid \eva_2]
&= \sum_{v \in V(G)} \E[W_v \mid \eva_2]
 = n \, O(\ln^2 n)
 = o(n^{1.1}) .
\end{align*}

Thus, in both the uniform and exponential cases,
\begin{align}
\E[X_k \mid \evp k]
&=\Pr(\eva_1) \E[X_k \mid \eva_1] + \Pr(\eva_2) \E[X_k \mid \eva_2]  \notag \\
&\leq (1-o(1)) \cdot (1+\eps) {\mu_k} + O(n^{-1.9}) \cdot o(n^{1.1}) \notag \\
&=(1-o(1))(1+\eps){\mu_k}  , \label{EXkUlarge}
\end{align}
since $\mu_k > 2k/n > n^{-6/10} = \omega(n^{-0.8})$.
As this holds for arbitrarily small $\eps > 0$, for all $k \geq n^{4/10}$,
\begin{align}
\E[X_k \mid \evp k]
  \leq (1+o(1)) \mu_k , \label{eq:expectUB}
\end{align}
completing the proof.

\section*{Acknowledgements}
We thank Alan Frieze and Wes Pegden for an initial discussion of the
second-shortest path,
and Alan for noticing that minimum-cost $k$-flow
(\cref{rmk:pathsFk}) was not an open problem
but immediately implied by our other results.
We also thank two anonymous referees for helpful suggestions.

\printbibliography
\end{document}